\documentclass[12pt,reqno,a4paper]{amsart}
\usepackage{amsmath,amssymb,amsthm,amsfonts}

\usepackage[draft]{changes}

\numberwithin{equation}{section}

{\theoremstyle{definition}
\newtheorem{defn}{Definition}[section]
}
\newtheorem{theorem}{Theorem}[section]
\newtheorem{proposition}[theorem]{Proposition}

\newtheorem{lemma}[theorem]{Lemma}

{\theoremstyle{definition}
{
\newtheorem{remark}[theorem]{Remark}

}}

\newcommand{\cal}{\mathcal}

\newcommand{\BB}{{\cal B}}
\newcommand{\CC}{{\cal C}}

\newcommand{\FF}{{\cal F}}
\newcommand{\GG}{{\cal G}}

\newcommand{\II}{{\cal I}}

\newcommand{\LL}{{\cal L}}

\newcommand{\RR}{{\cal R}}

\newcommand{\TT}{{\cal T}}

\newcommand{\UU}{{\cal U}}
\newcommand{\VV}{{\cal V}}

\newcommand{\fU}{{\mathfrak U}}

\newcommand{\Cc}{{\mathbb{C}}}

\newcommand{\Ee}{{\mathbb{E}}}

\newcommand{\Ii}{{\mathbb{I}}}

\newcommand{\Nn}{{\mathbb{N}}}

\newcommand{\Rr}{{\mathbb{R}}}

\newcommand{\Tt}{{\mathbb{T}}}

\newcommand{\Zz}{{\mathbb{Z}}}

\def\e{\mathrm{e}}

\def\diag{\operatorname{diag}}

\def\SL{\operatorname{SL}}

\def\SLZ{\SL(d,\Zz)}
\def\SLR{\SL(d,\Rr)}

\def\veck{{\text{$k$}}}

\def\vecv{{\text{$v$}}}
\def\vecw{{\text{$w$}}}
\def\vecx{{\text{$x$}}}

\def\vecalf{{\text{$\alpha$}}}

\def\vecomega{{\text{$\omega$}}}

\newcommand{\id}  {\operatorname{Id}}

\newcommand{\rot}{\operatorname{Rot}}
  
\newcommand{\te}[1]{\quad\text{#1}\quad}

\begin{document}

\title[]
{Renormalization of Gevrey vector fields with a Brjuno type arithmetical condition}

\author[J Lopes Dias]{Jo\~ao Lopes Dias}
\address{Departamento de Matem\'atica and CEMAPRE, ISEG\\ 
Universidade de Lisboa\\
Rua do Quelhas 6, 1200-781 Lisboa, Portugal}
\email{jldias@iseg.ulisboa.pt}

\author[Gaiv\~ao]{Jos\'e Pedro Gaiv\~ao}
\address{Departamento de Matem\'atica and CEMAPRE, ISEG\\
Universidade de Lisboa\\
Rua do Quelhas 6, 1200-781 Lisboa, Portugal}
\email{jpgaivao@iseg.ulisboa.pt}

\thanks{}
\date{\today}  

\begin{abstract}

We show that in the Gevrey topology, a $d$-torus flow close enough to linear with a unique rotation vector $\omega$ is linearizable as long as $\omega$ satisfies a Brjuno type diophantine condition.
The proof is based on the fast convergence under renormalization of the associated Gevrey vector field.
It requires a multidimensional continued fractions expansion of $\omega$, and the corresponding characterization of the Brjuno type vectors.
This demonstrates that renormalization methods deal very naturally with Gevrey regularity expressed in the decay of Fourier coefficients.
In particular, they provide linearization for frequencies beyond diophantine in non-analytic topologies. 
\end{abstract}
 
\maketitle


\section{Introduction}

The study of quasiperiodic motion yields a remarkable problem where dynamics, number theory and functional analysis meet intrisically.
It consists on the straightening of orbits, hoping that there are invariant sets which are essentially minimal translations with zero Lyapunov exponents.
It turns out that the existence and regularity of the corresponding coordinate change depends deeply on the arithmetical properties of the motion frequency.
This phenomenon relies on the subtile control of Fourier modes which are resonant with respect to the frequency, the so-called small divisors.

Flows on the torus provide one of the simplest but fundamental examples where to tackle small divisors problems.
The same ideas can be extended to more elaborated systems such as the Hamiltonian ones. 
The dimension plays also an important role in the type of results that can be obtained.
Indeed, the Poincar\'e transversal map of the equilibria-free two dimensional torus flow consists in a circle diffeomorphism whose theory was largely developed by Arnold~\cite{Arnold2}, Herman~\cite{Herman1979} and Yoccoz~\cite{Yoccoz2,Yoccoz4} in the real-analytic, smooth and finite regularity classes. 
On the other hand, in higher dimensions many questions remain unanswered besides the case of small perturbations around linear dynamics.
Those questions include the  optimality of the frequency conditions and non-perturbative results.

In this work we study vector fields on the $d$-torus $\Tt^d=\Rr^d/\Zz^d$, $d\geq2$, having Gevrey regularity.
The conjugacy class of a constant vector field $\omega$ depends on the arithmetical properties of $\omega$ and on the considered topology.
It is well-known that for real-analytic vector fields, if $\omega$ satisfies a Brjuno diophantine condition, the topological and real-analytic conjugacy classes coincide in some neighbourhood of $\omega$ (cf.~ \cite{R01}).
This property is known as rigidity, in the sense that the topology implies the geometry of the system.
Here we show that local rigidity also holds for Gevrey vector fields with rotation vector of a particular Brjuno type.
This shows that the diophantine condition is not optimal for Gevrey vector fields as in the smooth case, since the new class of vectors strictly contains all diophantines.

Functions with $s$-Gevrey regularity are, in a sense, an interpolation between real-analytic ($s=1$) and smooth ($s=\infty$) ones.
Their decay of Fourier coefficients behaves like $e^{-\rho|k|^{1/s}}$ where $\rho>0$. For $s=1$ this is the decay for analytic functions on a complex strip of width $\rho$.
A special feature for $s>1$ is that we can construct $s$-Gevrey bump functions.

An $s$-Brjuno vector is defined to be any $\omega\in\Rr^d$ such that
$$
\sum_{n\geq0}
\frac1{2^{n/s}}\max_{0<\|k\|\leq 2^n,k\in\Zz^d}\log\frac1{|k\cdot\omega|} <\infty.
$$
The classical Brjuno condition is given by $s=1$. 

\begin{theorem}\label{main thm I}
Let $s\geq1$.
If an $s$-Gevrey flow on $\Tt^d$ has a unique rotation $s$-Brjuno vector $\vecomega$ and it is $s$-Gevrey-close enough to linear, then it is $s$-Gevrey-conjugate to the torus translation $x\mapsto x+\omega t \bmod\Zz^d$, $t\geq0$.
\end{theorem}

Notice that if a vector field is topologically conjugate to $\omega$, then its rotation vector is unique and equal to $\omega$. Therefore, local rigidity follows from the above theorem.

We show Theorem~\ref{main thm I} using a renormalization method, taking advantage of the multidimensional continued fraction expansion of a vector in the spirit of Lagarias~\cite{Lagarias94} and Cheung~\cite{C11} (cf.~\cite{C13}).
The renormalization acts on the space of Gevrey vector fields and convergence to a trivial limit set implies conjugacy to a constant vector.
Requiring a sufficiently fast convergence rate restricts the class of frequencies, thus determining the $s$-Brjuno condition $\omega$ using continued fractions (see section~\ref{section:Multidimensional continued fractions}).

The above theorem also holds for the related problems of existence of invariant tori in Hamiltonian systems near integrable on $T^*\Tt^d$, including lower dimensional tori, and quasiperiodic linear skew-product flows on $\Tt^d\times \SL(d,\Rr)$. The proofs, to be detailed in a separate publication, are adaptations of the renormalization constructed in this work for Gevrey vector fields as is done in~\cite{jld5,jld9,jld8,jld7,MR2679012,Kocic2007} for the real-analytic class.
Moreover, the equivalent results for the discrete time version of all these systems are also achievable using similar methods.

Carletti and Marmi~\cite{MR1765828} studied the Siegel center problem~\cite{MR0007044} of one-dimensional germs of diffeomorphisms for ultradifferentiable classes including Gevrey.
In particular, they find that the Brjuno condition is sufficient to obtain linearization in this context. 
Other results on quasiperiodic systems in the Gevrey topology and Diophantine frequencies have only been obtained by analytic approximation techniques and using KAM methods ~\cite{MR2835876,MR2684068,MR2104602,MR2684071,MR2525200, MR2317497,MR2257153}, similarly to what is usually done for the finite differentiability case~\cite{MR0380867}. 
It is however a cumbersome strategy, with some obvious limitations when confronted with direct methods.
As shown in this work, the renormalization approach is naturally constructed for the Gevrey case, giving simpler and stronger results as it is capable of dealing with some Liouville frequencies.
Moreover, since the rescaling in the renormalization iteratively increases $\rho$, it avoids a common limitation while working with Gevrey and ultradifferentiable regularities related to estimates for the composition of functions (which have the effect of decreasing $\rho$).

The work of Koch~\cite{Koch} initiated a rigorous construction of renormalization operators on the space of real-analytic vector fields and Hamiltonian functions (cf. \cite{MacKay}). It was later improved by Khanin, Lopes Dias and Marklof~\cite{jld5,jld9} in order to deal with diophantine frequencies (see also \cite{Koch-Kocic06}) by incorporating multidimensional continued fractions.
Renormalization consists on rescaling space and reparametrizing time. Zooming into a region in phase space requires an acceleration of the orbits in order to detect self-similarity, a fixed point (or other simple orbits) of the renormalization.
Such fixed points are vector fields and can be trivial or critical.
The former corresponds to the scope of KAM theory, namely the stability of persistence of invariant tori. The latter is related to invariant tori on the verge of breakup, i.e. at the boundary of the domain of attraction of the trivial points. Evidence of this is harder to obtain, and it is mostly through the help of computer-assisted methods (cf.~\cite{MacKayThesis,Koch2004}).

Standard notations are included in section~\ref{section:Preliminaries} and section~\ref{section:Multidimensional continued fractions} presents the multidimensional continued fractions scheme and the derivation of the set of $s$-Brjuno vectors.
Section~\ref{section:gevrey} is on $s$-Gevrey functions.
Sections~\ref{section:Coordinates} and~\ref{section:Renormalization} define the renormalization operator, and section~\ref{section:Differentiable rigidity} includes the construction of the conjugacy for vector fields which are attracted under renormalization to the orbit of the constant system.

\section{Preliminaries}\label{section:Preliminaries}

We set the notations $\Nn=\{1,2,\dots\}$ for the positive integers and $\Nn_0=\Nn\cup\{0\}$ for the non-negative integers. 
The $\ell_1$-norm on $\Cc^d$ is denoted by
$$
|v|:=\sum_{i=1}^d|v_i|.
$$ 
The canonical inner product between vectors $u,v\in\Cc^d$ is given by 
$$
u\cdot v:=\sum_iu_iv_i
$$ 
and it satisfies
$$
|u\cdot v| \leq |u|\,|v|.
$$

Given a fixed constant $\mu>0$ (whose choice will be motivated later in section~\ref{sec: Sufficient conditions} and it will be defined in~\eqref{eq defn mu}), define the norm
$$
\|v\|_*:=\max\{ \|\hat v\|,|v_d|\}
\te{and}
\|\hat v\|:=\mu \sum_{i=1}^{d-1}|v_i|,
$$
where we use the notations $v=(\hat v,v_d)\in\Cc^d$ with 
$$
\hat v=(v_1,\dots,v_{d-1})\in\Cc^{d-1}.
$$

The above also defines the corresponding norm of a matrix $A=(a_{i,j})$ as the operator norm
\begin{align*}
|A| &= \sup_{|v|=1}|Av|=\max_j\sum_i|a_{i,j}|.\\
\end{align*}
The transpose matrix of $A$ is denoted by $A^\top$ and its inverse (if it exists) is writtten as
$$
A^{-\top}:=(A^\top)^{-1}.
$$
In addition,
$
|A^\top|  \leq d\, |A| .
$

\section{Multidimensional continued fractions}
\label{section:Multidimensional continued fractions}

We introduce here a multidimensional continued fractions expansion of vectors in $\Rr^d$ and its main properties related to renormalization.

\subsection{A special orbit on homogeneous spaces}

Consider the homogeneous space $\Gamma\backslash G$ with $G=\SLR$ and $\Gamma=\SLZ$, the space of $d$-dimensional unimodular lattices.
On its fundamental domain $\FF\subset G$ consider the right action of the one-parameter subgroup
$$
E^t=\diag(\e^{-t}, \dots, \e^{-t}, \e^{(d-1)t}) \in G
$$
that generates the flow
\begin{equation}\label{def geod flow}
\Phi^t\colon \FF \to\FF,
\qquad
M\mapsto \Gamma M E^t,
\end{equation}
This flow is known to be ergodic~\cite{Bekka2000}. 
In the following we will be interested in the properties of one particular orbit.

The size of the shortest non-zero vector in a lattice $M\in \FF$ is given by 
\begin{equation}
\delta\colon  \FF\to\Rr^+,
\qquad
\delta(M)=\inf_{\veck\in\Zz^d\setminus \{0\}} \| k^\top M\|_*.
\end{equation}
Notice that $\delta(\Phi^t (M))=\delta(M E^t)$ and that, due to Minkowski's theorem, there is some universal constant $\delta_0\geq1$ depending only on $d$ and the norm such that
$$
\delta(M)\leq \delta_0, 
\quad
M\in G.
$$

In the following fix $\vecomega=(\vecalf,1)\in\Rr^d$.
As we will see, the forward orbit $\Phi^t(M_0)$, $t\geq0$, of the matrix 
\begin{equation}
M_0=
\begin{pmatrix}
I & \alpha \\
0 & 1
\end{pmatrix}
\end{equation}
will present us many arithmetical properties of the vector $\omega$.
We have,
$$
\delta(\Phi^t(M_0))=\inf_{k\in\Zz^d\setminus\{0\}}
\max\left\{
e^{-t}\|\hat k\|, e^{(d-1)t}|k\cdot\omega|
\right\}.
$$

Define the map
$$
W\colon \Rr^+_0 \to \Rr,
\qquad
W(t)=\log\frac1{\delta(\Phi^t(M_0))}.
$$
So, $W(0)=0$ because $\delta(M_0)=1$.
In addition, $W(t)\geq -\log\delta_0$.

Notice that the function $W$ can be written as
\begin{equation*}
W(t)=
\sup_{k\in\Zz^d\setminus\{0\}} 
\Delta_k(t),
\end{equation*}
where we have the continuous piecewise functions for each $k$,
\begin{equation}\label{def Deltak}
\Delta_k(t)=\min\left\{
t- \log\|\hat{k}\| , -(d-1)t+\log\frac{1}{|k\cdot\omega|}
\right\}.
\end{equation}
The function $W$ is continuous since $\{\Delta_k\}_{k}$ is equicontinuous.

We observe that $\Delta_k(t)\leq t$ for any $t$. Indeed, the only case that is not immediate from~\eqref{def Deltak} is $\Delta_{(0,k_d)}(t)=-(d-1)t-\log|k_d| \leq -(d-1)t\leq t$ because $k_d\not=0$.

Moreover, we can write
$$
W(t)=\sup_{q\in\Nn}\sup_{\|\hat{k}\|=q}\Delta_k(t) =
\sup_{q\in\Nn}\Delta_{p(q)}(t)
$$
where $p(q)\in\Zz^d\setminus\{0\}$ is chosen such that
$$
\|\hat{p}(q)\|=q
\te{and}
|p(q)\cdot\omega| = \min_{\|\hat{k}\|=q}|k\cdot\omega|.
$$

We have  
$$
\Delta_{p(q)}(t)=
\begin{cases}
t- \log q ,& 0\leq t\leq T(q) \\
-(d-1)t+\log\frac{1}{|p(q)\cdot\omega|}, & 
t\geq T(q),
\end{cases}
$$
with 
$$
T(q)=\frac1d\log\frac{q}{|p(q)\cdot\omega|}.
$$

Take the sequence $q_0=1$ and for $n\in\Nn$ 
$$
q_{n}=
\inf
\left\{\|\hat k\|>0\colon k\in\Zz^d\setminus\{0\}, |k\cdot\omega| < |p(q_{n-1})\cdot\omega|
\right \}.
$$
Thus, $W$ is a continuous piecewise affine function with slopes either equal to 1 or $-(d-1)$ given by
$$
W(t)=\Delta_{p_n}(t),
\qquad
\tau_n\leq t\leq \tau_{n+1} 
$$
where $p_n=p(q_n)$ and
\begin{equation}\label{def tau_n}
\tau_n=\frac1d\log\frac{\|\hat p_n\|}{|p_{n-1}\cdot\omega|}.
\end{equation}
The terms in the ordered  sequence $\tau_n$ of the local minimizers of $W$,
$$
\tau_0=0 < \tau_1 < \tau_2 < \dots,
$$
are called \textit{stopping times}. 
Their number can be either finite or infinite.
The local maximizers of $W$ are 
$$
T_n:=T(q_n)=\frac1d\log\frac{\|\hat p_n\|}{|p_n\cdot\omega|}.
$$

Notice that
\begin{equation}\label{eq W - tau_n}
W(\tau_n)=\tau_n-\log \|\hat p_n\|=\frac1d\log\frac{1}{\|\hat p_n\|^{d-1}|p_{n-1}\cdot\omega|}
\end{equation}
and
\begin{equation}\label{eq: tau n+1}
\tau_{n+1}-W(\tau_{n+1})=\tau_n-W(\tau_n) + d(\tau_{n+1}-T_n).
\end{equation}
In addition,
\begin{equation}\label{eq: fn W(t)}
W(t)=\begin{cases}
t-(\tau_n-W(\tau_n)),& \tau_n\leq t\leq T_n\\
t-d(t-T_n)-(\tau_n-W(\tau_n)),& T_n<t\leq \tau_{n+1}.
\end{cases}
\end{equation}
It is also simple to check that
$$
W(t) \leq t-\log (n+1),
\quad
t \geq \tau_{n}
$$ 
for each $n\geq0$ such that $\tau_{n}$ exists.

\begin{lemma}\label{lem: rel pn and small div}
For any $n\in\Nn$,
$$
\|\hat p_n\| \leq \frac{\delta_0^{d/(d-1)}}{|p_{n-1}\cdot\omega|^{1/(d-1)}}.
$$
\end{lemma}

\begin{proof}
Recall that $W(\tau_n)\geq -\log \delta_0$ and that the difference between consecutive minima and maxima of $W$ is given by
$$
W(\tau_n)-W(T_{n-1})=-(d-1)(\tau_n-T_{n-1}),
\quad
n\in\Nn.
$$
Thus,
$$
W(T_{n-1})=T_{n-1}-\log\|\hat p_{n-1}\| \geq 
(d-1)(\tau_n-T_{n-1})-\log\delta_0
$$
and, by replacing the formulas of $\tau_n$ and $T_{n-1}$,
$$
\frac{d-1}{d}\log\|\hat p_{n}\| \leq 
\frac{1}{d}\log\frac{1}{|p_{n-1}\cdot\omega|}
+\log\delta_0.
$$
\end{proof}

\begin{proposition}[\cite{jld5}]\label{bnds Mn}
There exist $C_1,C_2>0$ such that for all $t\geq0$
\begin{align*}
\left|\Phi^t (M_0)\right| 
& \leq C_1 e^{(d-1)W(t)}\\
\left|\Phi^t (M_0)^{-1}\right| 
& \leq C_2e^{W(t)}.
\end{align*}
\end{proposition}

\subsection{Classification of vectors}

Recall that $\omega\in\Rr^d$ is \textit{rationally independent} (also called irrational) if $|k\cdot\omega|>0$ for every $k\in\Zz^d\setminus\{0\}$.
Otherwise it is called rationally dependent.
Moreover, $\omega$ is rationally independent iff 
$\{k\cdot\omega\colon k\in\Zz^d\}$ is dense in $\Rr$.

\begin{proposition}
$\omega$ is rationally independent iff there are infinite stopping times $\tau_n\to+\infty$.
\end{proposition}

\begin{proof}
Assume that there is an integer vector $k\not=0$ such that $k\cdot\omega=0$. 
Then, $W(t)=\Delta_k(t)=t-\log\|\hat k\|$ for every $t\geq \log\|\hat k\|$, which eliminates the possibility of infinite local minimizers.

On the other hand, if there are only finite local minimizers, take the largest one $\tau_n$.
Thus,  for $t>\tau_n$ the function $W$ has to be increasing  and thus equal to $t\mapsto t-\log\|\hat k\|$ for some integer vector $k$. That is only possible if $\log(1/|k\cdot\omega|)=+\infty$.
\end{proof}

\begin{lemma}\label{lem: k vs hat k}
If $\omega=(\alpha,1)\in\Rr^d$ and $k\in\Zz^d$, then
\begin{align*}
|k_d| & \leq |\alpha|\,|\hat k| +|k\cdot\omega| \\
\|k\|_* & \leq (|\alpha| + \mu)\,|\hat k| + |k\cdot\omega|.
\end{align*}
\end{lemma}

\begin{proof}
From the relation
$$
|k\cdot\omega| =|\hat k\cdot\alpha+k_d| \geq |k_d|-|\hat k\cdot\alpha|
\geq |k_d|-|\alpha|\,|\hat k|
$$
we obtain the first claim.
Finally,
$$
\|k\|_*=\max\{\|\hat k\|,|k_d|\} \leq \|\hat k\| + |k_d|=\mu|\hat k| + |k_d|.
$$
\end{proof}

Let $a\geq1$ and the sets of integer vectors given by
\begin{align*}
K_a &= \left\{ k\in\Zz^d\colon 0<\|k\|_*\leq a\right\} \\ 
\hat K_a &= \left\{ k\in\Zz^d\colon 0<\|\hat k\|\leq a\right\}.
\end{align*}

\begin{lemma}\label{lemma K hat K}
If $\omega=(\alpha,1)\in\Rr^d$, $a\geq\mu$ and $b=\max\{a,(a+1)|\alpha| \}$, then
$$
\min_{K_b}|k\cdot\omega| \leq 
\min_{\hat K_a}|k\cdot\omega| \leq
\min_{K_a}|k\cdot\omega|.
$$
\end{lemma}

\begin{proof}
Since $K_a\subset \hat K_a$ the second inequality follows immediately.
Now, notice that $\min_{\hat K_a} |k\cdot\omega|\leq |(1,0,\dots,0)\cdot\omega|\leq|\alpha|$.
Moreover, for any $k\in\hat K_a$, Lemma~\ref{lem: k vs hat k} implies that $|k_d|\leq (a+1)|\alpha|\leq b$.
As $\|\hat k\|\leq a\leq b$ we conclude that $\hat K_a\subset K_b$.
\end{proof}

\subsubsection{$s$-Brjuno vectors}

For $s\geq1$, a vector $\omega\in\Rr^d$ is $s$-Brjuno, i.e. $\omega\in BC(s)$, if 
$$
B_1(s):=\sum_{n\geq0}\frac1{2^{n/s}}\max_{0<\|k\|_*\leq 2^{n}}\log\frac{1}{|k\cdot\omega|} 
< \infty.
$$ 
Notice that the convergence (and divergence) of $B_1$ is independent of the norm used.
It follows that 
$$
BC(s)\subset BC(s')
\te{if}
s\geq s'\geq1.
$$
The case $s=1$ corresponds to the well-known Brjuno contition.
It is also clear that $\omega$ being $s$-Brjuno implies that all its coordinates are non-zero. Also, $\omega$ is $s$-Brjuno iff $c\omega$ is $s$-Brjuno with $c\not=0$, and this class of vectors is $\SL(d,\Zz)$-invariant.

Recall the sequence of vectors $p_n$ that correspond to the local minima of the function $W$ of a vector $\omega=(\alpha,1)$.

\begin{proposition}\label{prop:equivBrjuno}
Let $\omega=(\alpha,1)$ and $s\geq1$.
The following propositions are equivalent:
\begin{enumerate}
\item
$\omega\in BC(s)$.
\item
$$
B_2(s):=\sum_{n\geq0}\frac1{\|\hat{p_n}\|^{1/s}}\log\frac{1}{|p_{n}\cdot\omega|} 
< \infty.
$$ 
\item
$$
B_3(s):=\sum_{n\geq0} e^{-(\tau_n-W(\tau_n))/s}\tau_{n+1} <\infty.
$$
\end{enumerate}
\end{proposition}

\begin{proof}
By Lemma~\ref{lemma K hat K} we have that $B_1<\infty$ iff 
$$
B_1':=\sum_{n\geq0}\frac1{2^{n/s}}\max_{0<\|\hat k\|\leq 2^{n}}\log\frac{1}{|k\cdot\omega|}  < \infty.
$$

For each $n\in\Nn$ we can find $j_n\in\Nn$ such that 
$$
2^{j_n-1}\leq\|\hat p_n\| \leq 2^{j_n}.
$$
Notice then that $j_0=1$ and
\begin{align*}
\log\frac{1}{|p_n\cdot\omega|}
&= 
\max_{\|\hat k\| = \|\hat p_n\|}\log\frac{1}{|k\cdot\omega|} \\
&\leq
\max_{0<\|\hat k\|\leq 2^{j_n}}\log\frac{1}{|k\cdot\omega|}.
\end{align*}
Thus,
$$
B_2\leq 
\sum_{n\geq0}\frac1{2^{(j_n-1)/s}}\max_{0<\|\hat k\|\leq 2^{j_n}}\log\frac{1}{|k\cdot\omega|}
\leq
2^{1/s}B_1'.
$$

Choose now $i_n\in\Nn$ for each $n\in\Nn$ such that
$$
\|\hat p_{i_n}\|=\max\{\|\hat p_k\|\colon \|\hat p_k\|\leq 2^n,k\in\Nn\}.
$$
So,
\begin{align*}
B_1'
&\leq
\sum_{n\geq0}
\frac1{\|\hat p_{i_n}\|^{1/s}} \max_{\|\hat k\| = \|\hat p_{i_n}\|}\log\frac{1}{|k\cdot\omega|}  \leq
B_2.
\end{align*}

Using Lemma~\ref{lem: rel pn and small div} we get
\begin{align*}
B_3
&=
\frac1d
\sum_{n\geq0} \frac1{\|\hat p_n\|^{1/s}} \log\frac{\|\hat p_{n+1}\|}{|p_n\cdot\omega|} \\
&\leq
\frac1{d-1}\sum_{n\geq0} \frac1{\|\hat p_n\|^{1/s}}
\log\frac{ \delta_0 }{|p_n\cdot\omega|} \\
&=
\frac{1+\xi}{d-1}\sum_{n\geq0} \frac1{\|\hat p_n\|^{1/s}}
\log\frac{ 1}{|p_n\cdot\omega|} \\
&=
\frac{1+\xi}{d-1} B_2,
\end{align*}
where $\xi=-\log\delta_0/\log|p_0\cdot\omega|$ and we have used the fact that $|p_n\cdot\omega|\leq |p_0\cdot\omega|$.

Finally, by~\eqref{def tau_n} and~\eqref{eq W - tau_n}
\begin{align*}
B_2
&=\sum_{n\geq0} \frac1{\|\hat p_n\|^{1/s}} \log\frac{1}{|p_n\cdot\omega|} \\
&=
\sum_{n\geq0} e^{-(\tau_n-W(\tau_n))/s}\left(d\tau_{n+1} -\log\|\hat p_{n+1}\| \right) \\
&\leq 
d B_3.
\end{align*}
\end{proof}

\subsection{Contraction of orthogonal cones}\label{sec:contraction}

Consider any strictly increasing unbounded sequence $t_n>0$, $n\in\Nn$, and set $t_0=0$.
Let 
$$
M_n:=\Phi^{t_n}(M_0)
$$
a sequence of points in the orbit of $M_0$.
This is computed using a matrix $P_n\in\Gamma$ such that $M_n$ is in $\FF$.
That is,
$$
M_n=P_n M_0 E^{t_n}
=\begin{bmatrix}
\hat p^{(1)}_{n} e^{-t_n} & (p^{(1)}_{n}\cdot\omega) e^{(d-1)t_n} \\
\vdots & \vdots\\
\hat p^{(d)}_{n} e^{-t_n} & (p^{(d)}_{n}\cdot\omega) e^{(d-1)t_n}
\end{bmatrix}
$$
where $p^{(i)}_{n}=e_i^\top P_n$ is the $i$-th row of $P_n$ since $e_i$ is the $i$-th vector of the canonical basis of $\Rr^d$.
Moreover, set $P_0=I$.

The last column of $M_n$ is 
$$
\omega_n := M_n e_d =\lambda_n P_n \omega,\quad \lambda_n:= e^{(d-1)t_n}.
$$
Notice that $\omega_0=\omega$.
In addition, we define the matrices
$$
T_n:=P_nP_{n-1}^{-1} \in\Gamma
$$
so that $\omega_n=\eta_n T_n\omega_{n-1}$ with
$$
\eta_n=\frac{\lambda_n}{\lambda_{n-1}}
=e^{(d-1)(t_n-t_{n-1})}
$$
and $P_n=T_n\dots T_1$ for any $n\in\Nn$.

\begin{lemma}\label{mcf estimates}
For every $n\geq1$ the following holds:
\begin{enumerate}
\item
$|M_n|\leq C_1 e^{(d-1)W_n}$,
\item
$|M_n^{-1}|\leq C_2 e^{W_n},$
\item $|\omega_n|\leq C_1|\omega|\,e^{(d-1)W_n}$,
\item $|P_n|\leq C_1|\omega|\, e^{t_n-W_n+dW_n}$,
\item $|P_n^{-1}|\leq C_2|\omega|\, e^{(d-1)(t_n-W_n)+ dW_n}$,
\item $|T_n|\leq C_1C_2\,e^{t_n-W_n-(t_{n-1}-W_{n-1})+d W_n}$,
\item
$|T_n^{-1}|\leq C_1C_2\,e^{(d-1)(t_n-t_{n-1})+(d-1)W_{n-1}+W_n}$,
\end{enumerate}
where $C_1$ and $C_2$ are the constants in Proposition~\ref{bnds Mn} and $W_n=W(t_n)$.
\end{lemma}

\begin{proof}
This follows immediately from Proposition~\ref{bnds Mn}.
Notice that $T_n=M_nE^{-(t_n-t_{n-1})}M_{n-1}^{-1}$.
\end{proof}

The hyperbolicity of the matrices $T_n^{-\top}$, $n\geq1$, can be derived by looking at the contraction of the subspace 
$$
S_{n-1}^\perp=
\{ v\in\Rr^d\colon
v\cdot\omega_{n-1}=0 \}
$$
orthogonal to $\omega_{n-1}$. 

Denote by $\hat P_n$ the matrix $P_n$ with zeros on its last column.

\begin{lemma}
If $v\in S_{n-1}^\perp$, then
$$
|T_n^{-\top}v| \leq e^{-t_n} |M_n^{-\top}|\,|\hat P_{n-1}^{\top}| \, |v|.
$$
\end{lemma}

\begin{proof}
Firstly, since $\omega_{n-1}$ is given by the last column of $M_{n-1}$,  any $v\in S_{n-1}^\perp$ is orthogonal to it. 
Recall also that $T_n=P_nP_{n-1}^{-1}$. Thus, 
\begin{align*}
T_n^{-\top}v 
&=
P_n^{-\top}P_{n-1}^\top v \\
&=M_n^{-\top}E^{t_n}M_0^\top M_0^{-\top} E^{-t_{n-1}} M_{n-1}^\top v \\
&=M_n^{-\top}
\begin{bmatrix}
(\hat p^{(1)}_{n-1} )^\top e^{-t_n} & \dots & (\hat p^{(d)}_{n-1})^\top e^{-t_n}  \\
(p^{(1)}_{n-1}\cdot\omega) e^{(d-1)t_n} & \dots & (p^{(d)}_{n-1}\cdot\omega) e^{(d-1)t_n}
\end{bmatrix}
v \\
&=
e^{-t_n} M_n^{-\top}
\begin{bmatrix}
(\hat p^{(1)}_{n-1} )^\top & \dots & (\hat p^{(d)}_{n-1})^\top   \\
0 & \dots & 0
\end{bmatrix}
v.
\end{align*}
\end{proof}

Given a sequence $\sigma_{n}>0$ consider the following cones of integers vectors
$$
I_{n}^+ :=\{ k \in\Zz^d \colon 
|\omega_n\cdot k |\leq \sigma_{n} |k| \}
\te{and}
I_n^-:=\Zz^d\setminus I_n^+.
$$
We will refer the vectors in $I_n^+$ as resonant and in $I_n^-$ as far from resonant.
Let
\begin{align*}\label{def An}
A_n&=A_n(\sigma_{n-1},\omega_{n-1}):=\sup_{\veck\in I_{n-1}^+\setminus\{0\}}\frac{|T_n^{-\top}k|}{|k|}\,,\\
B_n&=B_n(\sigma_{n},\omega_{n}):=\sup_{\veck\in I_{n}^-}\frac{|P_n^{\top}\veck|}{|\veck|}.
\end{align*}

\begin{proposition}\label{lemma resonance cone}
For any $n\geq1$
\begin{equation}\label{formula An}
A_{n}
\leq
\frac{\sigma_{n-1} |\omega_{n-1}|}{\omega_{n-1}\cdot \omega_{n-1}}
|T_{n}^{-\top} | +
e^{-t_n} |M_{n}^{-\top} | \,|\hat P_{n-1}^\top|,
\end{equation}
where $C_1$ and $C_2$ are the constants in Proposition~\ref{bnds Mn}.
\end{proposition}

\begin{proof}
Any $\veck\in I_{n-1}^+\setminus\{0\}$ can be written as $\veck=\veck_1+\veck_2$ where
$$
\veck_1=\frac{\veck\cdot \vecomega_{n-1}}{\vecomega_{n-1}\cdot  \vecomega_{n-1}}\vecomega_{n-1}
\te{and}
\veck_2\in S_{n-1}^\perp.
$$
Hence,
\begin{equation}
\begin{split}
|T_{n}^{-\top} k| 
&\leq 
|T_{n}^{-\top} k_1| + 
|T_{n}^{-\top} k_2| \\
&\leq
\left(
\frac{\sigma_{n-1} |\omega_{n-1}|}{\omega_{n-1}\cdot \omega_{n-1}}
|T_{n}^{-\top} | +
e^{-t_n} |M_{n}^{-\top} | \,|\hat P_{n-1}^\top|
\right)
|\veck|.
\end{split}
\end{equation}
\end{proof}

Let
$$
\Delta(t)=\tau_{k(t)}-W(\tau_{k(t)})
$$
where
$ k(t)=\max\{j\in\Nn_0\colon \tau_j\leq t\}$.
Notice that
$$
\Delta(t_n)\leq t- W(t),\quad \tau_{k_n}\leq t <\tau_{k_n+1},
$$
where $k_n=k(t_n)$. Moreover $\Delta(t)$ is non-decreasing. Let $\Delta_n=\Delta(t_n)$. Thus,
\begin{equation}\label{estimate hat P_n}
|\hat P_n^\top| = \max_{i=1,\dots,d} |\hat p_n^{(i)}|
=|\hat p_n|=\mu^{-1} \|\hat p_n\| = \mu^{-1} e^{\Delta_n}
\end{equation}
by the fact that the first column on $M_n$ is always a best diophantine approximation~\cite{Lagarias94,C13} and~\eqref{eq W - tau_n}. 
\begin{lemma}\label{mcf estimates 2}
If for every $n\geq1$, $\xi_n>0$ and
$$
\sigma_n \leq \xi_n C_1^{-1}\mu^{-1} e^{-(d-1)(t_{n+1} -t_{n})-(d-1)W_n-t_{n+1}+\Delta_{n}},
$$
then
\begin{enumerate}
\item 
$A_n\leq (1+\xi_n)C_2\mu^{-1}\,e^{-\Delta_n+\Delta_{n-1}}$,
\item 
$A_1\cdots A_n\leq (1+\xi_n)^n C_2^n \mu^{-n}\,e^{-\Delta_n}$,
\item 
$A_1\cdots A_{n}B_{n}\leq |\omega|C_1(1+\xi_n)^n C_2^n\mu^{-n}  e^{-\Delta_n+t_n-W_n+d W_{n}}$,
\end{enumerate}
where $C_1$ and $C_2$ are the constants in Proposition~\ref{bnds Mn}.
\end{lemma}

\begin{proof}
It follows from Lemmas~\ref{mcf estimates}, Proposition~\ref{lemma resonance cone}, and \eqref{estimate hat P_n}.
Notice that we use the following relations: $\vecv\cdot\vecv\geq |\vecv|^2/d$ and $|\omega_n|\geq1$. 
\end{proof}

\section{Functional Spaces}
\label{section:gevrey}

\subsection{Gevrey spaces}

Let $\Tt^d=\Rr^d/(2\pi\Zz)^d$ with $d\geq2$.
The set of smooth $\Rr$-valued $2\pi\Zz^d$-periodic functions on $\Rr^d$ is denoted by $C^\infty(\Tt^d)$. In the following we shall use multi-index notation. So given $\alpha=(\alpha_1,\ldots,\alpha_d)\in\Nn_0^d$, where $\Nn_0=\{0,1,2,\dots\}$, we write
$$
\alpha!=\alpha_1!\cdots \alpha_d!\,,\quad |\alpha|=\alpha_1+\cdots+\alpha_d\quad\text{and}\quad \partial^\alpha=\partial^{\alpha_1}_{x_1}\cdots\partial^{\alpha_d}_{x_d}
$$
for the derivatives.
The sup-norm of $f\in C^\infty(\Tt^d)$ is defined as 
$$
\|f\|_{C^0}:=\max_{\vecx\in\Rr^d}|f(\vecx)|.
$$

\begin{lemma}
Let $a_n\geq0$, $n\in\Nn$, and $s\geq1$. 
Then,
\begin{enumerate}
\item
\begin{equation}\label{eq standard estimates2}
\sum_{i=1}^\infty a_i^s 
\leq \left(\sum_{i=1}^\infty a_i\right)^s 
\end{equation}
\item
\begin{equation}\label{eq standard estimates3}
\sum_{i=1}^d a_i^{1/s} 
\leq d^{(s-1)/s} \left(\sum_{i=1}^d a_i\right)^{1/s} 
\end{equation}
\end{enumerate}
\end{lemma}

\begin{proof}\hfill
\begin{enumerate}
\item
Assume that $0<\sum_{i=1}^\infty a_i<\infty$ (the remaining cases are immediate).
Thus,
$$
\sum_{j=1}^\infty\left(\frac{a_j}{\sum_{i=1}^\infty a_i}\right)^s\leq
\sum_{j=1}^\infty\frac{a_j}{\sum_{i=1}^\infty a_i} =1.
$$
\item

By the convexity of the function $x\mapsto x^s$,
$$
\left(\frac{b_1+\cdots+b_d}{d}\right)^s\leq \frac{b_1^s+\cdots+b_d^s}{d}.
$$
Now set $b_i=a_i^{1/s}$.
\end{enumerate}
\end{proof}

A smooth function $f\in C^{\infty}(\Tt^d)$ is \textit{$s$-Gevrey} with $s\geq 1$ if there exist constants $C>0$ and $\rho>0$ such that 
$$
\|\partial^\alpha f\|_{C^0}\leq C \frac{\alpha!^s}{\rho^{s|\alpha|}},
\quad
\alpha\in\Nn_0^d.
$$

Gevrey functions constitute an intermediate regularity class between smooth ($s=+\infty$) and real-analytic functions ($s=1$). 
Every $1$-Gevrey function is real-analytic because its Taylor series converges in a complex strip of radius $\rho$.

It is worthwhile observing that, unlike analytic functions, it is possible to construct $s$-Gevrey functions supported on any compact subset if $s>1$. 

\begin{remark}
The above definition of $s$-Gevrey function requires 
$$\|\partial^\alpha f\|_{C^0}\leq C L_{\alpha} M_{\alpha}$$ with $M_\alpha=\alpha!^s$ and $L_\alpha=\rho^{-s|\alpha|}$. Other sequences $M_\alpha$ give more general \textit{ultradifferentiable classes} (or \textit{Carleman classes}) of functions widely used in other branches of mathematics (see~\cite{MS02} and references therein).
\end{remark}

Fixing the constant $\rho>0$, Marco and Sauzin have defined the following spaces of Gevrey functions \cite{MS02}. A smooth function $f\in C^\infty(\Tt^d)$ belongs to $\CC_{s,\rho}(\Tt^d)$ if 
$$
\|f\|_{\CC_{s,\rho}}:=\sum_{\alpha\in\Nn^d_0}\frac{\rho^{s|\alpha|}}{\alpha!^s} \|\partial^\alpha f\|_{C^0}<\infty\,.
$$
The advantage of introducing this norm is that $\CC_{s,\rho}(\Tt^d)$ becomes a Banach algebra~\cite{MS02}. It is also clear that $\|f\|_{\CC_{s,\rho'}}\leq \|f\|_{\CC_{s,\rho}}$ for $0<\rho'<\rho$ and that any $s$-Gevrey function belongs to $\CC_{s,\rho}(\Tt^d)$ for some $\rho>0$. 
That is, the set of $s$-Gevrey functions is $\bigcup_{\rho>0}\CC_{s,\rho}(\Tt^d)$.
Moreover, we have the following Cauchy-type estimate.

\begin{lemma}[{\cite[Lemma A.2.]{MS02}}]\label{lm:cauchy}
If $0<\rho'<\rho$ and $f\in \CC_{s,\rho}(\Tt^d)$, then for every $\alpha\in\Nn_0^d$ the partial derivative $\partial^\alpha f$ belongs to $\CC_{s,\rho'}(\Tt^d)$ and
$$
\sum_{|\alpha|=n}\|\partial^\alpha f\|_{\CC_{s,\rho'}}\leq \frac{n!^s}{{(\rho-\rho')}^{ns}}\|f\|_{\CC_{s,\rho}}\,.
$$
\end{lemma}

Another important property of Gevrey functions is that the composition of Gevrey functions is again Gevrey. 

\begin{theorem}[{\cite[Corollary A.1.]{MS02}}]\label{th:MarcoSauzin}
If $0<d^{\frac{s-1}{s}}\rho'<\rho$, $f\in \CC_{s,\rho}(\Tt^d)$ and $u=(u_1,\ldots,u_d)$ with $u_i\in \CC_{s,\rho'}(\Tt^d)$ such that
$$
\|u_i\|_{\CC_{s,\rho'}}\leq \frac{\rho^s}{d^{s-1}}-\rho'^s,
$$ 
then $f\circ(\id+u)\in \CC_{s,\rho'}(\Tt^d)$ and $\|f\circ (\id + u)\|_{\CC_{s,\rho'}}\leq \|f\|_{\CC_{s,\rho}}$. 
\end{theorem}

Other interesting results about Gevrey functions can be found in \cite[Appendix A]{MS02}. See also~\cite{rodino}.


We denote by $C^\infty(\Tt^d,\Rr^d)$ the set of smooth $\Rr^d$-valued $2\pi\Zz^d$-periodic functions on $\Rr^d$. 
Given $f=(f_1,\dots,f_d)\in C^\infty(\Tt^d,\Rr^d)$ and $\rho>0$ we define the following $s$-Gevrey norm,
$$
\|f\|_{\CC_{s,\rho}} := \|f_1\|_{\CC_{s,\rho}} + \cdots + \|f_d\|_{\CC_{s,\rho}}.
$$
Similarly, we denote by $\CC_{s,\rho}(\Tt^d,\Rr^d)$ the set of $\Rr^d$-valued $s$-Gevrey functions that satisfy $\|f\|_{\CC_{s,\rho}}<\infty$, which is a Banach space.
Both Lemma~\ref{lm:cauchy} and Theorem~\ref{th:MarcoSauzin} hold for $\Rr^d$-valued $s$-Gevrey functions.


Given $f\in \CC_{s,\rho}(\Tt^d,\Rr^d)$, the derivative $Df$ can be seen as a continuous linear operator defined on $\CC_{s,\rho}(\Tt^d,\Rr^d)$. Denote by $\|Df\|_{\CC_{s,\rho}}$ the induced operator norm, i.e.
$$
\|Df\|_{\CC_{s,\rho}} = 
\max_{1\leq j\leq d} \sum_{i=1}^d 
\left\|\partial^{e_j}f_i\right\|_{\CC_{s,\rho}}
=
\max_{|\alpha|=1}
\left\|\partial^{\alpha}f\right\|_{\CC_{s,\rho}},
$$
where $e_j$ are the canonical basis vectors of $\Rr^d$.

Denote by $\CC'_{s,\rho}(\Tt^d,\Rr^d)$ the set of $s$-Gevrey functions that satisfy 
$$
\|f\|_{\CC'_{s,\rho}}
:=\|f\|_{\CC_{s,\rho}} + \|Df\|_{\CC_{s,\rho}}
<\infty.
$$ 
The set $\CC'_{s,\rho}(\Tt^d,\Rr^d)$ together with the norm $\|\cdot\|_{\CC'_{s,\rho}}$ is a Banach space contained in $\CC_{s,\rho}(\Tt^d,\Rr^d)$. 

Since $s$ will be fixed, in order to simplify the notation we shall write $\CC_{\rho}$ and  $\CC'_{\rho}$ in place of $\CC_{s,\rho}(\Tt^d,\Rr^d)$ and $\CC'_{s,\rho}(\Tt^d,\Rr^d)$, respectively.

\begin{lemma}\label{lem:norm estimates}
If $0<d^{\frac{s-1}{s}}\rho'<\rho$, $f\in \CC'_{\rho}$ and $u\in \CC_{\rho'}$ such that
$$
\|u\|_{\CC_{\rho'}}\leq \frac{\rho^s}{d^{s-1}}-\rho'^s,
$$ 
then
\begin{enumerate}
\item \label{estimate Df(id+u)}
$\|Df\circ (\id + u)\|_{\CC_{\rho'}}\leq\|Df\|_{\CC_{\rho}}$,
\item \label{estimate f(id+u)-f}
$\|f\circ(\id+u)-f\|_{\CC_{\rho'}}\leq \|Df\|_{\CC_{\rho}}\|u\|_{\CC_{\rho'}}$,
\end{enumerate}
Moreover, if 
$$
\|u\|_{\CC_{\rho'}}\leq \frac{(\rho+d^{\frac{s-1}{s}}\rho')^s}{2^sd^{s-1}}-\rho'^s,
$$ 
then
\begin{enumerate}
\item[(3)]
$
\|f\circ(\id + u)-f\|_{\CC_{\rho'}}\leq \frac{2^s}{(\rho-d^{\frac{s-1}{s}}\rho')^s}
\|f\|_{\CC_{\rho}}\|u\|_{\CC_{\rho'}}.
$
\item[(4)]\label{estimate Df(id+u)-Df}
$
\|Df\circ(\id + u)-Df\|_{\CC_{\rho'}}\leq \frac{2^s}{(\rho-d^{\frac{s-1}{s}}\rho')^s}
\|Df\|_{\CC_{\rho}}\|u\|_{\CC_{\rho'}}.
$
\end{enumerate}
\end{lemma}

\begin{proof}\hfill
\begin{enumerate}
\item
From the definitions of the norms and Theorem~\ref{th:MarcoSauzin} one gets
\begin{equation*}
\begin{split}
\|Df\circ(\id+u)\|_{\CC_{\rho'}}
&=
\max_{|\alpha|=1}
\left\|
(\partial^{\alpha}f)\circ(\id+u)\right\|_{\CC_{\rho'}} \\
&\leq
\max_{|\alpha|=1}
\left\|
\partial^{\alpha}f\right\|_{\CC_{\rho}}
=\|Df\|_{\CC_\rho} \\
&\leq
\|f\|_{\CC'_\rho}.
\end{split}
\end{equation*}
\item
Fix $x\in\Rr^d$ and write $g_i(t)=f_i(x+tu)$ with $g_i'(t)=Df_i(x+tu)\,u$. 
Then, $f_i(x+u)-f_i(x)=g_i(1)-g_i(0)=\int_0^1g_i'(t)\,dt$.
So,
$$
f(x+u)-f(x)=\int_0^1Df(x+tu)\, u\,dt.
$$
Using ~\eqref{estimate Df(id+u)} we obtain
\begin{equation*}
\begin{split}
\|f\circ(\id+u)-f\|_{\CC_{\rho'}} 
&\leq
\max_{0\leq t\leq1}\|Df\circ(\id+tu)\|_{\CC_{\rho'}} \|u\|_{\CC_{\rho'}} \\
&\leq
\|Df\|_{\CC_\rho}\|u\|_{\CC_{\rho'}}.
\end{split}
\end{equation*}
\item
The estimate~\eqref{estimate f(id+u)-f} with $\rho$ replaced by $\widetilde\rho:=(\rho+d^{\frac{s-1}{s}}\rho')/2 > d^{\frac{s-1}{s}} \rho'$ yields
$$
\|f\circ(\id + u)-f\|_{\CC_{\rho'}}\leq 
\|Df\|_{\CC_{\widetilde\rho}}\|u\|_{\CC_{\rho'}}
$$
and Lemma~\ref{lm:cauchy} implies that
$$
\|Df\|_{\CC_{\widetilde\rho}} \leq \frac{2^s}{(\rho-\widetilde\rho)^s}\|f\|_{\CC_\rho}.
$$
\item
By~\eqref{estimate f(id+u)-f} we get
\begin{equation*}
\begin{split}
\|Df\circ(\id + u)-Df\|_{\CC_{\rho'}}
&=
\max_{|\alpha|=1}
\|\partial^\alpha f\circ(\id + u)-\partial^\alpha f\|_{\CC_{\rho'}} \\
&\leq
\max_{|\alpha|=1}
\|D\partial^\alpha f\|_{\CC_{\widetilde\rho}} \|u\|_{\CC_{\rho'}} .
\end{split}
\end{equation*}
Finally, Lemma~\ref{lm:cauchy} implies that 
$$
\|D\partial^\alpha f\|_{\CC_{\widetilde\rho}}
=\max_{|\beta|=1}\|\partial^\beta\partial^\alpha f\|_{\CC_{\widetilde\rho}} 
\leq \frac{1}{\left(\rho-\widetilde\rho\right)^s} \|\partial^\alpha f\|_{\CC_\rho}.
$$
Therefore,
\begin{equation*}
\begin{split}
\|Df\circ(\id + u)-Df\|_{\CC_{\rho'}}
&\leq
\frac{2^s}{(\rho-d^{\frac{s-1}{s}}\rho')^s} \|Df\|_{\CC_{\rho}}\|u\|_{\CC_{\rho'}}.
\end{split}
\end{equation*}
\end{enumerate}
\end{proof}

\begin{lemma}\label{lemma Hn-id vf v1}
If for each $n\geq1$ we have $0<d^{\frac{s-1}{s}}\rho_{n}<\rho_{n-1}$ and $f_n-\id\in\CC_{\rho_n}$  such that
$$
\|f_n-\id\|_{\CC_{\rho_n}}\leq \frac{\rho_{n-1}^s}{d^{s-1}}-\rho_n^s,
$$ 
then
$$
\|f_1\circ\cdots\circ f_n -\id\|_{\CC_{\rho_n}}\leq \sum_{i=1}^n\|f_i-\id\|_{\CC_{\rho_i}}.
$$
\end{lemma}

\begin{proof}
By writing $\varphi_n=f_n-\id \in\CC_{\rho_n}$, it is simple to check that
$$
f_1\circ\cdots\circ f_n-\id=\varphi_n+(f_1\circ\cdots\circ f_{n-1}-\id)\circ(\id+\varphi_n).
$$
Thus, by Theorem~\ref{th:MarcoSauzin},
$$
\|f_1\circ\cdots\circ f_n-\id\|_{\CC_{\rho_n}}\leq \|\varphi_n\|_{\CC_{\rho_n}}+\|f_1\circ\cdots\circ f_{n-1}-\id\|_{\CC_{\rho_{n-1}}}.
$$
The claim follows immediately.
\end{proof}

\begin{lemma}\label{lemma Hn-Hn-1 v1}
If for each $n\geq1$ we have $0<d^{\frac{s-1}{s}}\rho_{n}<\rho_{n-1}$ and $f_n-\id\in\CC_{\rho_n}$ such that
$$
\|f_n-\id\|_{\CC_{\rho_n}}
\leq 
\frac{(\rho_{n-1}+d^{\frac{s-1}{s}}\rho_n)^s}{2^sd^{s-1}}-\rho_n^s,
$$
then 
\begin{equation}\label{estimate Hn-Hn-1 v1}
\|f_1\circ\cdots\circ f_{n}-f_1\circ\dots\circ f_{n-1}\|_{\CC_{\rho_{n}}}
\leq
\left(1+ 
\frac{2^s}{(\rho_{n-1}-d^{\frac{s-1}{s}}\rho_n)^s}
\sum_{i=1}^{n-1}
\|f_i-\id\|_{\CC_{\rho_i}}
\right)
 \|f_{n}-\id\|_{\CC_{\rho_n}}.
\end{equation}
\end{lemma}

\begin{proof}
Write $h_n=f_1\circ\cdots\circ f_{n}$ and $\varphi_n=f_n-\id \in\CC_{\rho_n}$ for any $n\geq1$.
It is simple to check that
\begin{equation*}
\begin{split}
h_{n}-h_{n-1}
&=
\varphi_{n} + 
\sum_{i=1}^{n-1}\left(
\varphi_i\circ F_{i,n}
-\varphi_i\circ F_{i,n-1}
\right),
\end{split}
\end{equation*}
where 
$$
F_{m_1,m_2}:=f_{m_1+1}\circ\cdots\circ f_{m_2},\quad m_1< m_2,
$$
and $F_{m,m}=\id$.
Clearly, $F_{i,n}=F_{i,n-1}\circ f_n$.
So,
$$
h_{n}-h_{n-1}=
\varphi_{n} + 
\sum_{i=1}^{n-1}\left(
\varphi_i\circ F_{i,n-1}\circ(\id+\varphi_n)
-\varphi_i\circ F_{i,n-1}
\right).
$$

From Lemma~\ref{lem:norm estimates}, 
$$
\|\varphi_i\circ F_{i,n}-\varphi_i\circ F_{i,{n-1}}\|_{\CC_{\rho_n}}
\leq 
\frac{2^s}{(\rho_{n-1}-d^{\frac{s-1}{s}}\rho_n)^s}
\|\varphi_i\circ F_{i,n-1}\|_{\CC_{\rho_{n-1}}}
\|\varphi_{n}\|_{\CC_{\rho_n}}.
$$
Since $\varphi_i\circ F_{i,n-1}=\varphi_i\circ F_{i,n-2}\circ(\id+\varphi_{n-1})$,
by Theorem~\ref{th:MarcoSauzin},
\begin{equation*}
\begin{split}
\|\varphi_i\circ F_{i,n-1}\|_{\CC_{\rho_{n-1}}}
&\leq 
\|\varphi_i\circ F_{i,n-2}\|_{\CC_{\rho_{n-2}}} \\
&\leq
\|\varphi_i\|_{\CC_{\rho_{i}}}.
\end{split}
\end{equation*}

Finally,
$$
\|h_{n}-h_{n-1}\|_{\CC_{\rho_n}}
\leq
\left(1+ 
\frac{2^s}{(\rho_{n-1}-d^{\frac{s-1}{s}}\rho_n)^s}
\sum_{i=1}^{n-1}
\|\varphi_i\|_{\CC_{\rho_i}}
\right)
 \|\varphi_{n}\|_{\CC_{\rho_n}}.
$$
\end{proof}

\subsection{Decay of Fourier coefficients}

Any $f\in C^{\infty}(\Tt^d,\Rr^d)$ can be represented in Fourier series as 
$$
f(\vecx)=\sum_{\veck\in\Zz^d}f_{\veck}e^{i \veck\cdot\vecx},
$$
where
$$
f_k=\frac{1}{(2\pi)^d}\int_{\Tt^d}f(\vecx)e^{-ik\cdot x}\,d\vecx.
$$
We write the constant Fourier mode of $f$ through the projection
\begin{equation}
\Ee f=f_0.
\end{equation}

Let $|\veck|=|k_1|+\cdots+|k_d|$ for $\veck\in\Zz^d$. The following is a well-known estimate, we include here a proof only for the convenience of the reader. 

\begin{lemma}[Decay of Fourier coefficients]\label{lem:decay}
If $f\in\CC_{\rho}$, then 
$$
|f_\veck|\leq \Delta\|f\|_{\CC_{\rho}}e^{-\rho|\veck|^{1/s}},\quad\veck\in\Zz^d,
$$
where 
$$
\Delta:=(2\pi)^{-d}\left(1-s^{-\frac{s}{s-1}}\right)^{-(s-1)d}<\left(\frac{e}{2\pi}\right)^d.
$$
\end{lemma}
\begin{proof}
Since $|(\partial^\alpha f)_\veck|\leq \frac{1}{(2\pi)^d}\|\partial^\alpha f\|_{C^0}$ we have
$$
\|f\|_{\CC_{\rho}}=\sum_{\alpha\in\Nn_0^d}\frac{\rho^{s|\alpha|}}{\alpha!^s}\|\partial^\alpha f\|_{C^0}\geq (2\pi)^d\sum_{\alpha\in\Nn_0^d}\frac{\rho^{s|\alpha|}}{\alpha!^s}|(\partial^\alpha f)_\veck|\,.
$$
Taking into account that $(\partial^\alpha f)_\veck=\prod_j( i k_j)^{\alpha_j}f_\veck$ we get,
\begin{equation}\label{estimate fourier coefficient}
(2\pi)^d|f_\veck|\sum_{\alpha\in\Nn_0^d}\frac{\rho^{s|\alpha|}}{\alpha!^s}\veck^\alpha\leq \|f\|_{\CC_{\rho}},
\end{equation}
where $\veck^\alpha=\prod_{j=1}^d |k_j|^{\alpha_j}$.
Now we estimate the sum $\sum_{\alpha}\frac{\rho^{s|\alpha|}}{\alpha!^s}\veck^\alpha$ from below. Noticed that
$$
\sum_{\alpha\in\Nn_0^d}\frac{\rho^{s|\alpha|}}{\alpha!^s}\veck^\alpha=\prod_{j=1}^d\sum_{n=0}^\infty\frac{\rho^{sn}}{n!^s} |k_j|^{n}\,.
$$
In order to estimate the sum inside the product we recall the H\"{o}lder inequality. For any sequences of positive real numbers $(x_n)_{n\geq0}$ and $(y_n)_{n\geq0}$ we have
$$
\left(\sum_{n=0}^\infty x_n y_n\right)^s\leq  \left(\sum_{n=0}^\infty y_n^t\right)^{s/t}\sum_{n=0}^\infty x_n^s,
$$
where $t=\frac{s}{s-1}$. Taking 
$$
x_n=\frac{\left(\rho  |k_j|^{1/s} \right)^{n}}{n!}\quad\text{and}\quad y_n=\frac{1}{s^n}
$$
we get
$$
e^{\rho|k_j|^{1/s}}\leq h(s)\sum_{n=0}^\infty\frac{\rho^{sn}}{n!^s} |k_j|^{n},
$$
where $h(s):=\left(1-s^{\frac{s}{1-s}}\right)^{1-s}$. Notice that $h'(s)>0$, $h(1)=1$ and $\lim_{s\to\infty}h(s)=e$. So  $1\leq  h(s)< e$ for every $s\geq 1$.
Since $\sum_j|k_j|^{1/s}\geq |\veck|^{1/s}$, we get
$$
\sum_{\alpha\in\Nn_0^d}\frac{\rho^{s|\alpha|}}{\alpha!^s}\veck^\alpha\geq \frac{1}{h^d} e^{\rho|\veck|^{1/s}}.
$$
Using this lower bound in \eqref{estimate fourier coefficient} we obtain the desired estimate on the Fourier coefficients.
\end{proof}

\begin{lemma}\label{lem:mode}
For every $\rho>0$ and $\veck\in\Zz^d$, 
$$\|e^{i\veck\cdot\vecx}\|_{\CC_{\rho}}\leq e^{ d^{\frac{s-1}{s}}s\rho|\veck|^{\frac{1}{s}}}\quad\text{and}\quad\|e^{i\veck\cdot\vecx}\|'_{\CC_{\rho}}\leq (1+|\veck|)e^{ d^{\frac{s-1}{s}}s\rho|\veck|^{\frac{1}{s}}}.$$
\end{lemma}

\begin{proof}
We will prove only the first inequality. The second follows directly from the first and the definition of the norm. Notice that,
\begin{align*}
\|e^{i\veck\cdot\vecx}\|_{\CC_\rho}&=\sum_{\alpha\in\Nn_0^d}\frac{\rho^{s|\alpha|}}{\alpha!^s}\prod_{j=1}^d|k_j|^{\alpha_j}\\
&= \sum_{n=0}^\infty \sum_{|\alpha|=n}\prod_{j=1}^d\frac{|k_j|^{\alpha_j}}{\alpha_j!^s}\rho^{s\alpha_j}\\
&=\prod_{j=1}^d\sum_{n=0}^\infty \left(\frac{|k_j|^{\frac{n}{s}}}{n!}\rho^{n}\right)^s\\
&\leq
\left(\prod_{j=1}^d\sum_{n=0}^\infty \frac{|k_j|^{\frac{n}{s}}}{n!}\rho^{n}\right)^s \\
&\leq e^{s\rho\sum_j|k_j|^{\frac{1}{s}}},
\end{align*}
where we have used the fact
\begin{equation*}\label{eq standard estimates1}
\begin{split}
\sum_{n=0}^\infty \sum_{|\alpha|=n}\prod_{i=1}^d a_i(\alpha_i) 
=
\prod_{i=1}^d 
\sum_{n=0}^\infty a_i(n) 
\end{split}
\end{equation*}
for any sequences $a_i$  and\eqref{eq standard estimates2}.

Since $\sum_j|k_j|^{\frac{1}{s}}\leq d^{\frac{s-1}{s}}|\veck|^\frac{1}{s}$ from~\eqref{eq standard estimates3} we obtain the claimed estimate.
\end{proof}

\subsection{Spaces $\FF_{s,\rho}$ and $\FF'_{s,\rho}$}
Lemma~\ref{lem:decay} motivates the following definition. Given $\rho>0$ let $\FF_{s,\rho}(\Tt^d)$ be the set of smooth functions $f\in C^{\infty}(\Tt^d)$ that satisfy
$$
\|f\|_{\FF_{s,\rho}}:=\sum_{\veck\in\Zz^d}|f_\veck|e^{\rho|\veck|^{1/s}}<\infty\,.
$$
Several properties of this norm are easy to establish. Firstly, $\|f\|_{C^0}\leq \|f\|_{\FF_{s,\rho}}$ for any $\rho>0$. Secondly, $\|f\|_{\FF_{s,\rho'}}\leq\|f\|_{\FF_{s,\rho}}$ for every $\rho'<\rho$. Moreover, it is simple to check that ${\FF_{s,\rho}}(\Tt^d)$ endowed with the norm $\|\cdot\|_\rho$ is a Banach algebra.

Given any $f=(f_1,\dots,f_d)\in C^\infty(\Tt^d,\Rr^d)$ we define the following norm,
$$
\|f\|_{\FF_{s,\rho}} := \|f_1\|_{\FF_{s,\rho}} + \cdots + \|f_d\|_{\FF_{s,\rho}}.
$$
Similarly, we denote by $\FF_{s,\rho}(\Tt^d,\Rr^d)$ the set of $\Rr^d$-valued $s$-Gevrey functions that satisfy $\|f\|_{\FF_{s,\rho}}<\infty$. Clearly, $\FF_{s,\rho}(\Tt^d,\Rr^d)$ is also a Banach space.

To control the derivatives of Gevrey functions it is convenient to introduce the following family of norms,
$$
\|f\|_{\FF'_{s,\rho}}:=\sum_{\veck\in\Zz^d}(1+|\veck|)|f_k|e^{\rho|\veck|^{1/s}}
$$
and define $\FF'_{s,\rho}(\Tt^d,\Rr^d)\subset\FF_{s,\rho}(\Tt^d,\Rr^d)$ to be the subset of Gevrey functions that have the above norm finite. Notice that,
$$
\|f\|_{\FF'_{s,\rho}}=\|f\|_{\FF_{s,\rho}}+\sum_{|\alpha|=1}\|\partial^\alpha f\|_{\FF_{s,\rho}}.
$$

To simplify the notation we shall denote these spaces by $\FF_{s,\rho}$ and $\FF'_{s,\rho}$, and when there is no need for the explicit dependence of $s$ we remove it from our notation.

It is clear that $\FF'_{\rho}$ is also a Banach space. Moreover,
$$
\|Df(h)\|_{\FF_{\rho}}\leq \|f\|_{\FF'_{\rho}}\|h\|_{\FF_{\rho}}.
$$ 
This means that $Df$ is a bounded operator on $\FF_{\rho}$ whenever $f\in\FF'_{\rho}$. We also denote by $\|Df\|_{\FF_{\rho}}$ its induced norm. 

Another useful property is the following upper-bound on the norm of the derivatives of a function. 

\begin{lemma}[Cauchy's estimate] \label{lem:Cauchy}Given $\rho'<\rho$ and $f\in \FF_{\rho}$, 
$$
\|\partial^\alpha f\|_{\FF_{\rho'}}\leq \left(\frac{d^{\frac{s-1}{s}}s}{\rho-\rho'}\right)^{s|\alpha|}\alpha!^s\|f\|_{\FF_{\rho}}.
$$
\end{lemma}

\begin{proof}
Note that 
\begin{align*}
\|\partial^\alpha f\|_{\FF_{\rho'}}&=\sum_{\veck\in\Zz^d}|f_\veck|\prod_{j=1}^d |k_j|^{\alpha_j}e^{\rho'|\veck|^{1/s}}\\
&\leq\sum_{\veck\in\Zz^d}|f_\veck|\left(\prod_{j=1}^d |k_j|^{\alpha_j}e^{-d^{\frac{1-s}{s}}(\rho-\rho')|k_j|^{1/s}}\right)e^{\rho|\veck|^{1/s}}
\end{align*}
where we have used the inequality $d^{\frac{s-1}{s}}|\veck|^{1/s}\geq |k_1|^{1/s}+\cdots+|k_d|^{1/s}$. The function $x\mapsto x^{\alpha_j}e^{-d^{\frac{1-s}{s}}(\rho-\rho')x^{1/s}}$ defined for $x\geq0$ attains its maximum at $x^*=\left(\frac{\alpha_jd^{\frac{s-1}{s}}s}{\rho-\rho'}\right)^{s}$ with value $\left(\frac{\alpha_jd^{\frac{s-1}{s}}s}{e(\rho-\rho')}\right)^{s\alpha_j}$. Since $(\alpha_j/e)^{\alpha_j}\leq \alpha_j!$ by Stirling's approximation, we get
\begin{align*}
\|\partial^\alpha f\|_{\FF_{\rho'}} &\leq\prod_{j=1}^d  \left(\frac{\alpha_j d^{\frac{s-1}{s}}s}{e(\rho-\rho')}\right)^{s\alpha_j} \sum_{\veck\in\Zz^d}|f_\veck|e^{\rho|\veck|^{1/s}}\\
&\leq \left(\frac{d^{\frac{s-1}{s}}s}{\rho-\rho'}\right)^{s|\alpha|}\alpha!^s\|f\|_{\FF_{\rho}}\,.
\end{align*}
\end{proof}

In the following lemma we show how the norms of the various Banach spaces are related.
To simplify the notation we define the constants:
\begin{equation}\label{def: beta C_nu}
\beta:=d^{\frac{s-1}{s}}s\quad\text{and}\quad C_\nu:=\sum_{k\in\Zz^d}e^{-\nu|\veck|^{1/s}},
\end{equation}
where $\nu>0$ \footnote{Notice that $C_\nu$ can be bounded from above as follows,
$$
C_\nu\leq 1+ \left(\frac{\pi^2}{3}\right)^d\left(\frac{\beta}{\nu}\right)^{2sd}.
$$}.

\begin{lemma}[Inclusions]\label{lem:equivalence_norms}
Let $\rho'>0$ and $\nu>0$. The following holds:
\begin{enumerate}
\item If $\rho\geq\beta\rho'+\nu$, then 
$$
\|f\|_{\CC_{\rho'}}\leq \|f\|_{\FF_{\rho}}
\te
{and}\|f\|_{\CC'_{\rho'}}\leq \|f\|_{\FF'_{\rho}}.
$$
\item If $\rho\geq\rho'+\nu$, then
$$
\|f\|_{\FF_{\rho'}}\leq C_\nu\|f\|_{\CC_{\rho}}\te{and}\|f\|_{\FF'_{\rho'}}\leq C_\nu\|f\|_{\CC'_{\rho}}.
$$
\end{enumerate}
\end{lemma}

\begin{proof}
By Lemma~\ref{lem:mode}, we have
$$
\|f\|_{\CC_{\rho'}}\leq \sum_{\veck\in\Zz^d}|f_\veck|\|e^{i\veck\cdot\vecx}\|_{\CC_{\rho'}}\leq \sum_{\veck\in\Zz^d}|f_\veck|e^{(\beta\rho'+\nu)|\veck|^{1/s}}=\|f\|_{\FF_{\beta\rho'+\nu}}.
$$
This proves the first inequality of (1). Using Lemma~\ref{lem:decay} we get
$$
\|f\|_{\FF_{\rho'}}= \sum_{\veck\in\Zz^d}|f_\veck|e^{\rho'|\veck|^{1/s}}\leq C_\nu \|f\|_{\CC_{\rho'+\nu}},
$$
which shows the first inequality of (2). The remaining inequalities are proved similarly.
\end{proof}

\begin{remark}
It follows from the previous lemma that the set of $s$-Gevrey functions is given by $\bigcup_{\rho>0}\FF_{s,\rho}$. 
\end{remark}

\begin{proposition}\label{prop:norm estimates}
Given $\rho>0$ and $0<\nu<\rho/(1+\beta+\beta^2)$ let
$$
\rho':=\frac{\rho-\nu}{\beta}\quad\text{and}\quad \rho'':=\frac{\rho'-\nu}{\beta}-\nu.
$$
If $f\in\FF_{\rho}$ and $u\in\FF_{\rho'}$ such that
$$
\|u\|_{\FF_{\rho'}}\leq
\frac{\rho'^s}{d^{s-1}}-(\rho''+\nu)^s,
$$
then
\begin{enumerate}
\item $\|f\circ(\id+u)\|_{\FF_{\rho''}}\leq C_\nu\|f\|_{\FF_{\rho}}$,
\item $\|Df\circ(\id+u)\|_{\FF_{\rho''}}\leq C_{\nu}\|f\|_{\FF'_{\rho}}$,
\item $\|f\circ(\id+u)-f\|_{\FF_{\rho''}}\leq  C_{\nu}\|f\|_{\FF'_{\rho}}\|u\|_{\FF_{\rho'}}$,
\end{enumerate}
Moreover, if $f\in\FF'_{\rho}$ and,
$$
\|u\|_{\FF_{\rho'}}\leq \frac{\left(\rho'+d^{\frac{s-1}{s}}(\rho''+\nu)\right)^s}{2^sd^{s-1}}-(\rho''+\nu)^s
$$
then
$$\|Df\circ(\id + u)-Df\|_{\FF_{\rho''}}\leq \frac{2^s C_\nu }{\nu^s} \|f\|_{\FF'_{\rho}}\|u\|_{\FF_{\rho'}}.$$
\end{proposition}

\begin{proof}\hfill
\begin{enumerate}
\item
By (2) of Lemma~\ref{lem:equivalence_norms}, 
$$
\|f\circ(\id+u)\|_{\FF_{\rho''}}\leq C_\nu\|f\circ(\id+u)\|_{\CC_{\rho''+\nu}}.
$$
Since, by (1) of Lemma~\ref{lem:equivalence_norms},
$$
\|u\|_{\CC_{\rho''+\nu}}\leq \|u\|_{\FF_{\rho'}}\leq \frac{\rho'^s}{d^{s-1}}-(\rho''+\nu)^s,
$$ 
we get by Theorem~\ref{th:MarcoSauzin} and (1) of Lemma~\ref{lem:equivalence_norms} that,
$$
\|f\circ(\id+u)\|_{\FF_{\rho''}}\leq C_\nu \|f\|_{\CC_{\rho'}}\leq C_\nu\|f\|_{\FF_{\rho}}.
$$
\item Similarly, by (2) of Lemma~\ref{lem:equivalence_norms}, 
$$
\|Df\circ(\id+u)\|_{\FF_{\rho''}}\leq C_\nu\|Df\circ(\id+u)\|_{\CC_{\rho''+\nu}}.
$$
Thus, by \eqref{estimate Df(id+u)} of Lemma~\ref{lem:norm estimates},
$$
\|Df\circ(\id+u)\|_{\FF_{\rho''}}\leq C_\nu \|Df\|_{\CC_{\rho'}}\leq C_\nu \|f\|_{\FF'_{\rho}}.
$$
\item Arguing as before we conclude using Lemma~\ref{lem:equivalence_norms} and \eqref{estimate f(id+u)-f} of Lemma~\ref{lem:norm estimates} that,
\begin{equation*}
\begin{split}
\|f\circ(\id+u)-f\|_{\FF_{\rho''}}&\leq  C_{\nu}\|f\circ(\id+u)-f\|_{\CC_{\rho''+\nu}}\\
&\leq C_\nu \|Df\|_{\CC_{\rho'}}\|u\|_{\CC_{\rho''+\nu}}\\
&\leq C_\nu \|f\|_{\FF'_{\rho}}\|u\|_{\FF_{\rho'}}.
\end{split}
\end{equation*}
\end{enumerate}
To prove the last estimate we can apply \eqref{estimate Df(id+u)-Df} of Lemma~\ref{lem:norm estimates} to get
\begin{equation*}
\begin{split}
\|Df\circ(\id + u)-Df\|_{\FF_{\rho''}}&\leq C_\nu \|Df\circ(\id + u)-Df\|_{\CC_{\rho''+\nu}}\\
&\leq C_\nu  \frac{2^s}{(\rho'-d^{\frac{s-1}{s}}(\rho''+\nu))^s}
\|Df\|_{\CC_{\rho'}}\|u\|_{\CC_{\rho''+\nu}}\\
&\leq \frac{2^s C_\nu }{\nu^s}\|f\|_{\FF'_{\rho}}\|u\|_{\FF_{\rho'}}
\end{split}
\end{equation*}

\end{proof}


Since $\FF'_{\rho}\subset\FF'_{\rho-\log\phi}$ whenever $\phi\geq1$, consider the inclusion operator $\II_{\phi}\colon \FF'_{\rho}\to \FF'_{\rho-\log\phi}$. 
Notice that $\II_\phi\circ\Ee=\Ee\circ\II_\phi=\Ee$.
When restricted to non-constant modes, its norm can be estimated as follows.

\begin{lemma}\label{lemma cutoff}
If $\phi\geq1$, then $\|\II_\phi (\Ii-\Ee)\| \leq \phi^{-1}$.
\end{lemma}

\begin{proof}
This follows simply by noticing that
$$
\|(\Ii-\Ee)f\|_{\FF'_{\rho-\log\phi}} 
= 
\sum_{\veck\not=0}(1+|k|) |f_k|e^{(\rho-\log\phi)|k|^{1/s}}
\leq 
\phi^{-1}
\|f\|_{\FF'_{\rho}}
$$
\end{proof}

\label{section:Renormalization}

\section{Coordinate transformations and time reparametrization}
\label{section:Coordinates}

A coordinate transformation $\phi$ on the $d$-torus $\Tt^d$ is a diffeomorphism isotopic to a matrix in $SL(d,\Zz)$. That is, $\psi=\phi\circ A$ where $A\in SL(d,\Zz)$ and $\phi\colon\Tt^d\to\Tt^d$ is an isotopic to the identity diffeomorphism, meaning that $\phi-\id$ is $2\pi\Zz^d$-periodic.

A vector field $X$ on $\Tt^d$ written on new coordinates $\psi$ is denoted by
$$
\psi^*X=(D\psi)^{-1}X\circ \psi.
$$
Notice that the set of vector fields on $\Tt^d$ can be identified with the set of functions from $\Tt^d$ to $\Rr^d$, i.e. $2\pi\Zz^d$-periodic maps of $\Rr^d$. 

Since $s\geq1$ is fixed throughout the paper and only the Banach spaces $\FF_{\rho}$ and $\FF'_{\rho}$ will be used, we shall simplify the notation by denoting their norms by $\|\cdot\|_\rho$ and $\|\cdot\|'_\rho$, respectively.


\subsection{Elimination of far from resonance modes}

Fix $\vecw\in\Rr^d$.
Given $\sigma>0$ we call {\em far from resonance modes} to the Fourier modes with indices in 
\begin{equation}\label{def I-}
I_{\sigma,\vecw}^- =\left\{ \veck\in \Zz^d \colon 
|\vecw\cdot \veck| > \sigma |\veck| 
\right\}.
\end{equation}
The {\em resonant modes} are the ones in $I_{\sigma,\vecw}^+=\Zz^d\setminus I_{\sigma,\vecw}^-$.
We also define the projections $\Ii_{\sigma,\vecw}^+$ and $\Ii_{\sigma,\vecw}^-$ over the spaces of functions by restricting the modes to $I_{\sigma,\vecw}^+$ and $I_{\sigma,\vecw}^-$, respectively.
Clearly, $\Ii=\Ii_{\sigma,\vecw}^++\Ii_{\sigma,\vecw}^-$ where $\Ii$ is the identity operator. Moreover, $\|\Ii^\pm_{\sigma,\vecw}\|_{\rho}\leq 1$.
To simplify the notation we occasionally omit the dependence of $I^\pm_{\sigma,\vecw}$ and $\Ii^\pm_{\sigma,\vecw}$ from $\vecw$.

Given $\rho>0$ and $\varepsilon>0$, we denote by $\VV_\varepsilon$ the set
\begin{equation}\label{def V ball}
\VV_\varepsilon=\{\vecw+f\in \FF'_{\rho} \colon  \|f\|'_\rho<\varepsilon\}.
\end{equation}

The following theorem is an adaptation of a result in \cite{jld,jld5} to the Gevrey class. For the convenience of the reader a proof can be found in the appendix.

\begin{theorem}
\label{thm unif}
Given $0<\sigma<|\vecw|$, $\rho>0$ and $0<\nu<\rho/(1+\beta+\beta^2)$,  let
\begin{equation}\label{eq:espilon}
\varepsilon=\varepsilon(\sigma,\nu,|\vecw|,s,d):=\frac{\sigma}{8(C_\nu-1)}\min\left\{\frac{\nu^s}{(2\beta)^s},\frac{\sigma}{8|\vecw| C_\nu}\left(\frac{2^s}{\nu^s}+7\right)^{-1}\right\},
\end{equation}
and
$$
\rho':=\frac{\rho-\nu}{\beta}\quad\text{and}\quad \rho'':=\frac{\rho'-\nu}{\beta}-\nu.
$$
There exist a smooth homotopy of Fr\'echet differentiable maps $\fU_t\colon  \VV_{\varepsilon} \to \Ii_\sigma^-\FF'_{\rho'}$ and $\UU_t\colon \VV_{\varepsilon} \to  (1-t)\Ii^-_\sigma\FF'_{\rho}\oplus t\Ii^+_\sigma\FF_{\rho''}$ such that 
$$
\UU_t(X)=(\id+\fU_t(X))^*X
$$ 
and
\begin{equation}\label{equation homotopy method}
\Ii_\sigma^-\UU_t(X)=(1-t)\,\Ii_\sigma^-X,
\qquad
t\in[0,1].
\end{equation}
Moreover,
\begin{equation}\label{estimate U-id}
\|\fU_t(X)\|'_{\rho'}
\leq \frac{8t(C_\nu-1)}{\sigma}\|\Ii_\sigma^-X\|_{\rho},
\end{equation}
and 
\begin{equation}\label{estimate U around X0}
\begin{split}
\|\UU_t(X)-\vecw\|_{\rho''}
\leq &
\|\Ii^+_\sigma (X-\vecw)\|_{\rho''}+(1-t)\|\Ii^-_\sigma X\|_{\rho''}\\
&+\frac{2^9t|\vecw|(C_\nu-1)(2C_\nu-1)}{\sigma^2}{\|X-\vecw\|'_\rho}^2.
\end{split}
\end{equation}
\end{theorem}

\begin{remark}\label{remark:estimate U around X0}
It follows from the definition of $\varepsilon$ and estimate \eqref{estimate U around X0} that,
$$
\|\UU_t(X)-\vecw\|_{\rho''}\leq (8-t)\|X-\vecw\|'_\rho.
$$
\end{remark}

\subsection{Rescaling}

A fundamental step in the renormalization scheme is a linear transformation of the
domain of definition of our vector fields.

Suppose that $T\in \SL(d,\Zz)$ and $\eta\in\Rr\setminus\{0\}$. 
Consider $X\in\FF_{\rho}$.
We are interested in the following coordinate and time linear changes:
\begin{equation}\label{coord x and time change}
\vecx\mapsto T^{-1}\vecx,
\qquad
t\mapsto \eta t.
\end{equation}
Notice that $\eta<0$ means inverting the direction of time. These changes determine a new vector field as the image of the map
$$
X\mapsto\TT(X):=\eta\, (T^{-1})^* X .
$$
It is simple to check that $\Ee\circ\TT=\TT\circ\Ee$.

Let $|T|$ denote the induced norm of the matrix $T$, i.e. $$|T|=\max_{1\leq j\leq d}\sum_{i=1}^d|T_{i,j}|$$ where $T_{i,j}$ is the $i,j$ entry of $T$.
Clearly, $|T|\in\Nn$.

Given $\sigma>0$ and $w\in\Rr\setminus\{0\}$, define
$$
A:=\sup_{\veck\in I_{\sigma,\vecw}^+\setminus\{0\}}\frac{|(T^{\top})^{-1}\veck|}{|\veck|}.
$$

\begin{lemma}\label{proposition TT}
Let $\rho>0$, $0<\delta<\rho/A^{1/s}$ and
\begin{equation}\label{hyp on rho n' for vf}
\rho':=\frac{\rho}{A^{1/s}}-\delta.
\end{equation}
The linear operator $\TT(\Ii^+_{\sigma,\vecw}-\Ee)$ maps $\FF_{\rho}$ into $(\Ii-\Ee)\FF'_{\rho'}$ and satisfies
\begin{equation}\label{bdd tilde LL n}
\|\TT(\Ii^+_{\sigma,\vecw}-\Ee)\|\leq
|\eta|\, |T|\,
\left(1+\frac{s^s}{\delta^s}\right).
\end{equation}
\end{lemma}

\begin{proof}
Let $f\in(\Ii^+_{\sigma,\vecw}-\Ee)\FF_{\rho}$.
Then,
$$
\|f\circ T^{-1}\|_{\FF'_{\rho'}}
\leq
\sum_{\veck\in I^+_{\sigma,\vecw}\setminus\{0\}}
\left(1+|(T^{\top})^{-1}\,\veck|\right)
|f_\veck|
\e^{(\rho'-\delta+\delta)|(T^{\top})^{-1}\veck|^{1/s}}.
$$
Using the inequality $\xi \e^{-\delta\,\xi^{1/s}}\leq \left(\frac{s}{\delta}\right)^s$ with $\xi\geq0$, we get
\begin{align*}
\|f\circ T^{-1}\|_{\FF'_{\rho'}}
&\leq
\left(1+\frac{s^s}{\delta^s}\right)
\sum_{\veck\in I^+_{\sigma,\vecw}\setminus\{0\}} 
|f_\veck|
\e^{A^{1/s}(\rho'+\delta)|\veck|^{1/s}}\\
&\leq 
\left(1+\frac{s^s}{\delta^s}\right)
\|f\|_{\FF_{\rho}}.
\end{align*}
Finally, $\|\TT f\|_{\FF'_{\rho'}}\leq |\eta|\,|T| \,\|f\circ T^{-1}\|_{\FF'_{  \rho'}}$.
\end{proof}

Given $P\in \SL(d,\Zz)$, $\sigma>0$ and $w\in\Rr\setminus\{0\}$, define
$$
B:=\sup_{\veck\in I_{\sigma,\vecw}^-}\frac{|P^{\top}\veck|}{|\veck|}.
$$

\begin{lemma}\label{proposition TT v2}
Let $\rho>0$ and
\begin{equation*}
\rho':=\frac{\rho}{B^{1/s}}.
\end{equation*}
The linear operator $\tau\colon f\mapsto f\circ P$ maps $\Ii^-_{\sigma,w}\FF_{\rho}$ into $(\Ii-\Ee)\FF_{\rho'}$ and satisfies
$\|\tau\circ \Ii^-_{\sigma,w}\|\leq1$.
\end{lemma}

\begin{proof}
Let $f\in\Ii^-_{\sigma,\vecw}\FF_{\rho}$.
Then,
$$
\|f\circ P\|_{\FF_{\rho'}}
=
\sum_{\veck\in I^-_{\sigma,\vecw}}
|f_\veck|
\e^{\rho'|P^{\top}\veck|^{1/s}}
\leq
\sum_{\veck\in I^-_{\sigma,\vecw}} 
|f_\veck|
\e^{B^{1/s}\rho'|\veck|^{1/s}}
=
\|f\|_{\FF_{\rho}}.
$$
\end{proof}
%
\section{Renormalization}
\label{section:Renormalization}

As in the previous section, $s\geq1$ is fixed throughout and to simplify the notation we shall denote by  $\|\cdot\|_\rho$ and $\|\cdot\|'_\rho$ the norms of $\FF_{\rho}$ and $\FF'_{\rho}$, respectively.

\subsection{Renormalization operator}

Fix $\rho>0$.
Let $w\in\Rr^d\setminus\{0\}$, $\sigma>0$, $0<\nu<\rho/(1+\beta+\beta^2)$, $\eta\in\Rr\setminus\{0\}$ and $T\in SL(d,\Zz)$.
Recall also ~\eqref{def: beta C_nu}.
The renormalization operator 
$$
\RR\colon \FF'_{\rho} \to \bigcup_{r>0}\FF_r
$$ 
is defined for each $X\in\VV_\varepsilon$ by
$$
\RR(X)= \TT\circ \UU (X).
$$

\begin{proposition}\label{prop:one iteration}
Let $0<\delta<\rho''/A^{1/s}$, 
$$
\rho'=\frac{\rho''}{A^{1/s}}-\delta
\te{and}
\rho'' = \frac{\rho-\nu(1+\beta+\beta^2)}{\beta^2}.
$$
For any $X\in\VV_\varepsilon$ and $1\leq\phi<e^{\rho'}$ we have that $\RR(X)\in\FF'_{\rho'}$ and
\begin{equation*}
\begin{split}
\|(\Ii-\Ee)\RR(X)\|'_{\rho'-\log\phi} 
\leq &
\frac{|\eta|\, |T|}{\phi}\left( 1+\frac{s^s}{\delta^s} \right) 
\left[\vphantom{\frac{2^9|w|(C_\nu-1)(2C_\nu-1)}{\sigma^2}}
\|\Ii^+_{\sigma,w} (X-w)\|_{\rho''} 
\right. \\
& \left.
+\frac{2^9|w|(C_\nu-1)(2C_\nu-1)}{\sigma^2}{\|X-w\|'_\rho}^2
\right]
\end{split}
\end{equation*}

\end{proposition}

\begin{proof}
Using Theorem~\ref{thm unif}, Lemma~\ref{proposition TT} and Lemma~\ref{lemma cutoff} we obtain the above statement.
\end{proof}


\subsection{Infinitely renormalizable vector fields}
\label{The Limit Set of the Renormalization}

For a rationally independent vector $\vecomega\in\Rr^d\setminus\{0\}$ consider its multidimensional continued fractions expansion, namely the sequences $\vecomega_n$, $T_n$ and $\eta_n$, $n\geq1$.
Moreover, consider some chosen sequences $\rho_n,\sigma_n,\nu_n>0$ satisfying
$$
\sigma_n<|\omega_n|
\te{and}
\nu_n<\rho_n/(1+\beta+\beta^2).
$$

We now define a sequence of renormalization operators $\RR_n$ in the following way. Each renormalization operator is the conposition of the operators $\TT_n:=\eta_n\,(T_n^{-1})^*$ and $\UU_n$ obtained by Theorem~\ref{thm unif} for $t=1$ and $\vecw=\vecomega_{n-1}$, i.e.
$$
\RR_n:=\TT_n\circ\UU_n\,,\quad n\geq1.
$$

The domain of the operator $\RR_n$ is the open ball $\VV_{\varepsilon_{n-1}}\subset\FF'_{\rho_{n-1}}$ centered at $\omega_{n-1}$ with radius
$$
\varepsilon_{n-1}=\varepsilon(\sigma_{n-1},\nu_{n-1},|\omega_{n-1}|,s,d)
$$
as given by \eqref{eq:espilon}.
Notice that $X$ and $\RR_n(X)$ are Gevrey-equivalent vector fields, i.e. their flows are conjugated by an $s$-Gevrey diffeomorphism.

\begin{defn}
We say that $X\in\FF'_{\rho}$ is \textit{infinitely renormalizable} if  $X$ belongs to the domain of the operator $\RR_n\circ\cdots\circ \RR_1$ for every $n\geq1$,
i.e.
$$
\|X_{n-1}-\omega_{n-1}\|_{\rho_{n-1}} <\varepsilon_{n-1}.
$$
\end{defn}

We will show later that infinitely renormalizable vector fields such that the renormalization converges to a constant have a flow which is linearizable by a Gevrey conjugacy.
In the remaining part of this section we want to find conditions for which a vector field is infinitely renormalizable. 

\subsection{Sufficient conditions}

Let $\rho_0:=\rho$. We fix the sequence $\nu_n:=\nu>0$ to be  constant along the iterations and so that
\begin{equation}\label{nu constant}
\nu < \frac{\rho_n}{1+\beta+\beta^2}
\end{equation}
for every $n\geq0$.
This can be achieved for the choice 
\begin{equation}\label{rhon}
\rho_n:=\frac{\rho_{n-1}-\nu(1+\beta+\beta^2)}{\beta^2A_n^{1/s}} -\delta-\log\phi_n
\end{equation}
for any sequence $\phi_n\geq1$ and $\delta>0$, as long as $\inf_n\rho_n>0$.
Iterating the equation above we get
$$
\rho_n = \frac{\rho-\BB_n}{\beta^{2n} A_1^{1/s}\dots A_n^{1/s}}
$$
where
\begin{equation}\label{eq defn BB_n}
\BB_n :=\sum_{i=1}^n \beta^{2i} A_1^{1/s}\dots A_i^{1/s} 
\left(
\delta+\frac{\nu(1+\beta+\beta^2)}{\beta^2A_i^{1/s}}+
\log\phi_i
\right)
\end{equation}
is an increasing sequence.
Define
$$
\phi_n := \max\left\{7(d+1)|\eta_n|\, |T_n|  \,\left( 1+\frac{s^s}{\delta^s} \right) \frac{\varepsilon_{n-1}}{\varepsilon_n\theta_n},1\right\}
$$
 where $0< \theta_n \leq1$ is any chosen sequence.

Notice that $\BB_n$ depends on the choice of the sequence $\sigma_n$ through the sequences $\varepsilon_n$ and $A_n$. Moreover, if for some sequence $\sigma_n$ we have $\lim\BB_n<\infty$, then necessarily $\beta^{2n}A_1^{1/s}\cdots A_n^{1/s}\to 0$.
Hence, if $\rho>\lim\BB_n$, we have
$$
\rho_n > 
\frac{\rho-\lim \BB_n}{\beta^{2n} A_1^{1/s}\dots A_n^{1/s}} \to+\infty
$$

Let $X_0:=X$ and $X_n:=\RR_n(X_{n-1})$ whenever $X_{n-1}$ is in the domain of $\RR_n$. 

\begin{theorem}\label{convergence Rn}
If $X\in\FF'_{\rho}$, $0<\theta_n\leq1$ and $0<\sigma_n<|\omega_n|$ satisfy
\begin{itemize}
\item
$\rot X=\vecomega$, 
\item
$\|X-\omega\|'_\rho < \varepsilon_0$,
\item
$\rho > \lim\BB_n$,
\end{itemize}
then $X$ is infinitely renormalizable and
\begin{equation}\label{bound on LLn less than epsilon}
\|X_n-\vecomega_n\|'_{\rho_n}
< \varepsilon_{n} \theta_{n},
\qquad
n\geq1.
\end{equation}
\end{theorem}

\begin{proof}
If at each step $X_n$ is in the domain of $\UU_{n+1}$, i.e.
\begin{equation}\label{cdn inf ren}
\|X_{n}-\vecomega_n\|'_{\rho_n} <\varepsilon_n\,,
\end{equation} 
then $X_n$ is renormalizable and $X_{n+1}=\RR_{n+1}(X_n)$. Being true for any $n\in\Nn$, then $X$ is infinitely renormalizable.
The inequality \eqref{cdn inf ren} can be estimated using~\cite[Proposition 3.3]{jld10} and Proposition~\ref{prop:one iteration}. First we get,
\begin{equation}
\begin{split}
\|X_{n}-\vecomega_n\|'_{\rho_n} 
&=
\|\RR_n(X_{n-1})-\vecomega_n\|'_{\rho_n}
\\
&\leq
\|(\Ii-\Ee)\RR_n(X_{n-1})\|'_{\rho_n} + |\Ee\RR_n(X_{n-1})-\vecomega_n| \\
&\leq
(d+1) \|(\Ii-\Ee)\RR_n(X_{n-1})\|'_{\rho_n}.
\end{split}
\end{equation}
Thus,
\begin{equation*}
\begin{split}
\|X_{n}-\vecomega_n\|'_{\rho_n}  
\leq &
(d+1)\frac{|\eta_n|\, |T_n|}{\phi_n}\left( 1+\frac{s^s}{\delta^s} \right) 
\left[\vphantom{\frac{2^9|\vecomega_{n-1}|(C_\nu-1)(2C_\nu-1)}{\sigma_{n-1}^2}}
\|\Ii^+_{n-1} (X_{n-1}-\vecomega_{n-1})\|_{\xi'} 
\right. \\
& \left.
+\frac{2^9|\vecomega_{n-1}|(C_\nu-1)(2C_\nu-1)}{\sigma_{n-1}^2}{\|X_{n-1}-\vecomega_{n-1}\|'_{\xi}}^2
\right],
\end{split}
\end{equation*}
where 
$$
\xi'=A_n^{1/s}\left(\rho_n+\delta+\log\phi_n\right)\te{and}\xi=\beta^2\xi'+\nu\left(1+\beta+\beta^2\right).
$$

We now proceed by induction. Assuming that \eqref{bound on LLn less than epsilon} holds for $n-1$, we substitute the value of $\phi_n$ and use Remark~\ref{remark:estimate U around X0} to get,
$$
\|X_{n}-\vecomega_n\|'_{\rho_n}\leq 7(d+1)\frac{|\eta_n|\, |T_n|}{\phi_n}\left( 1+\frac{s^s}{\delta^s} \right)\|X_{n-1}-\vecomega_{n-1}\|'_{\xi}<\varepsilon_{n}\theta_n.
$$
\end{proof}


\section{Conjugacy to torus translation}
\label{section:Differentiable rigidity}
In this section we give a sufficient condition for a conjugacy of the flow of $X$ to a torus translation to have Gevrey regularity.

\subsection{Convergence of the conjugation}

Fix $s\geq1$ and let $\rho>0$. Assume that $X$ is infinitely renormalizable, i.e.
$$
\|X_n-\omega_n\|_{\FF'_{\rho_n}}\leq\varepsilon_n,
\qquad
n\geq1.
$$  
Notice that
\begin{equation}
X_n=\lambda_n\,(U_1 \circ T_1^{-1}\circ \cdots \circ U_{n}\circ T_n^{-1})^*(X)\in\FF'_{\rho_n},
\end{equation}
with $U_n:=\id+\fU_n(X_{n-1})\in\FF'_{\rho_{n-1}}$ and $\lim_n\rho_n=\infty$. 
Furthermore, we can write
\begin{equation}\label{formula X n with Vs}
P_n^*X_n= \lambda_n\, h_{n}^*(X)
\end{equation}
by considering the $s$-Gevrey diffeomorphisms
\begin{equation}\label{def H n}
h_{n}:=g_1\circ\dots\circ g_n
\end{equation}
and
\begin{equation*}
g_{n}:=
P_{n-1}^{-1}\circ U_{n} \circ P_{n-1}\,,\quad n\geq 1.
\end{equation*}

For convenience of notations, set $T_0=P_0=I$ to be the identity matrix. Notice that $|I|=1$.

Define,
$$
r_n:=
\frac{\rho_{n-1}-\nu}{2\beta^2 B_{n-1}^{1/s}}
,
\qquad
n\geq1.
$$
We recall that $0<\nu<\rho_n/(1+\beta+\beta^2)$. 

\begin{lemma}\label{lemma Wn-id}
For every $n\geq1$,
\begin{equation}\label{norm of Wn-I in Drho n+1}
\|g_n-\id\|_{\CC_{r_n}}
\leq 
8(C_{\nu}-1)\frac{|P_{n-1}^{-1}| }{\sigma_{n-1}}\|\Ii_{n-1}^-X_{n-1}\|_{\FF_{\rho_{n-1}}}.
\end{equation}
\end{lemma}

\begin{proof}
Lemma~\ref{lem:equivalence_norms}, Lemma~\ref{proposition TT v2}  and Theorem~\ref{thm unif} imply that
\begin{equation*}
\begin{split}
\|P_{n-1}^{-1} \circ U_n\circ P_{n-1}\|_{\CC_{r_n}} 
&\leq
|P_{n-1}^{-1}|\,\|U_n\circ P_{n-1}\|_{\FF_{\beta r_n+\mu_n}} \\
&\leq
|P_{n-1}^{-1}|\,\|U_n\|_{\FF_{\zeta_n}} \\
&\leq
\frac{8(C_{\nu}-1)\,|P_{n-1}^{-1}|}{\sigma_{n-1}} \, \|\Ii_{n-1}^-X_{n-1}\|_{\FF_{\rho_{n-1}}},
\end{split}
\end{equation*}
where 
$$
\zeta_n=B_{n-1}^{1/s}(\beta r_n+\mu_n)
\te{and}\mu_n=\frac{\rho_{n-1}-\nu}{2\beta B_{n-1}^{1/s}}.$$
\end{proof}

Given $\ell\in\Nn_0$, denote by $C^\ell(\Tt^d,\Rr^d)$ the space of $2\pi\Zz^d$-periodic function which have $\ell$ continuous derivatives. We consider the $C^\ell$-norm,
$$
\|f\|_{C^\ell}:=\sup_{|\alpha|\leq \ell}\|\partial^\alpha f\|_{C^0}.
$$
Also define,
$$
\Theta_n:=\frac{|P_{n-1}^{-1}| }{\sigma_{n-1}}\|X_{n-1}-\vecomega_{n-1}\|_{\FF_{\rho_{n-1}}}.
$$

From now on we consider a sequence of positive real numbers $\{R_n\}_{n\geq0}$ satisfying,
\begin{equation}\label{eq defn R_n}
R_n\leq r_n\te{and} d^{\frac{s-1}{s}}R_{n}<R_{n-1},\quad n\geq1.
\end{equation}

\begin{theorem}[Topological conjugacy]\label{th: top conjugacy}
If
$$
\sum_{n=1}^\infty
\frac{\Theta_n}{(R_{n-1}-d^{\frac{s-1}{s}}R_n)^s} <\infty,
$$
then $h:=\lim_nh_n$ is a homeomorphism and $\phi_X^t\circ h=h\circ \phi_\omega^t$.
\end{theorem}

\begin{proof}
Notice that, by Lemma~\ref{eq standard estimates2},
$$
\frac{(R_{n-1}+d^{\frac{s-1}{s}}R_n)^s}{2^sd^{s-1}}-R_n^s\geq \frac{\left(R_{n-1}-d^{\frac{s-1}s}R_n\right)^s}{2^sd^{s-1}}.
$$
The convergence of the series in the hypothesis implies that $$\lim_n\Theta_n/(R_{n-1}-d^{(s-1)/s}R_n)^s=0.$$ Thus, for $n$ sufficiently large we have
\begin{equation}\label{eq:condition}
\frac{8(C_{\nu}-1)|P_{n-1}^{-1}|}{\sigma_{n-1}}
\,
\|\Ii_{n-1}^-X_{n-1}\|_{\FF_{\rho_{n-1}}}
\leq
\frac{(R_{n-1}+d^{\frac{s-1}{s}}R_n)^s}{2^sd^{s-1}}-R_n^s.
\end{equation}
This condition is sufficient to apply Lemma~\ref{lemma Hn-Hn-1 v1}. So we get,
\begin{equation*}\label{estimate Hn-Hn-1 f}
\|h_{n}-h_{n-1}\|_{\CC_{R_{n}}}
\leq
\frac{\Gamma_n}{(R_{n-1}-d^{\frac{s-1}{s}}R_n)^s}
 \|g_{n}-\id\|_{\CC_{R_{n}}},
\end{equation*}
where
$$
\Gamma_n:=(R_{n-1}-d^{\frac{s-1}{s}}R_n)^s+ 
2^s
\sum_{i=1}^{n-1}
\|g_i-\id\|_{\CC_{R_i}}.
$$
Follows from Lemma~\ref{lemma Wn-id}, the properties of $R_n$ and $\sum_n\Theta_n<\infty$ that
$\Gamma:=\sup_n\Gamma_n<\infty$.
So, 
\begin{equation*}
\begin{split}
\|h_n-h_{n-1}\|_{\CC_{R_{n}}}
&\leq
\frac{\Gamma}{(R_{n-1}-d^{\frac{s-1}{s}}R_n)^s}\|g_{n}-\id\|_{\CC_{R_{n}}}.
\end{split}
\end{equation*}
Using again Lemma~\ref{lemma Wn-id}, we have
\begin{equation}\label{estimate h_n-h_n-1}
\|h_n-h_{n-1}\|_{\CC_{R_{n}}}\leq 8(C_\nu-1)\Gamma\frac{\Theta_n}{(R_{n-1}-d^{\frac{s-1}{s}}R_n)^s}.
\end{equation}
Noticed that $\|h_n-h_{n-1}\|_{C^0}\leq \|h_n-h_{n-1}\|_{\CC_{s,R_{n}}}$. Thus, $h_n-\id$ is a Cauchy sequence in $C^0$. Hence, it converges to $h-\id\in C^0(\Tt^d,\Rr^d)$ where $h:=\lim_n h_n$. To show that $h$ is a homeomorphism we prove that the inverse $h_n^{-1}$ also converges in $C^0$. Notice that,
$$
\|h_{n}^{-1}-h_{n-1}^{-1}\|_{C^0}=\|g_{n}^{-1}-\id\|_{C^0},
$$ 
and 
$$
\|g_n^{-1}-\id\|_{C^0}=\|(g_n-\id)\circ g_{n}^{-1}\|_{C^0}=\|g_n-\id\|_{C^0}.
$$
Thus,
$$
\|h_{n}^{-1}-h_{n-1}^{-1}\|_{C^0}=\|g_n-\id\|_{C^0}\leq \|g_{n}-\id\|_{\CC_{R_{n}}}.
$$
It follows immediately that $h_n^{-1}$ converges in $C^0$. Thus $h$ is a homeomorphism. 

Finally, we show that $h$ conjugates the flow of $X$ to a linear flow with frequency $\omega$. First notice that
$$
\phi^t_X\circ h_n = h_n \circ \phi^t_{\lambda_n^{-1}P_n^{*}X_n}.
$$
Since $\lambda_n^{-1}P_n^{*}X_n=\omega+\lambda_n^{-1}P_n^*(X_n-\omega_n)$
we get,
\begin{align*}
\left\|\phi^t_{\lambda_n^{-1}P_n^{*}X_n}-\phi^t_{\omega}\right\|_{C^0}&\leq\left\|\int_0^t\lambda_n^{-1}P_n^*(X_n-\omega_n)\circ \phi^s_{\lambda_n^{-1}P_n^{*}X_n}\,ds\right\|_{C^0}\\
&\leq |t|\lambda_n^{-1}|P_n^{-1}|\|X_n-\omega_n\|_{C^0}\\
&= |t|\sigma_{n}\lambda_n^{-1}\Theta_{n+1}\\
&\leq  |t||\omega_n|\lambda_n^{-1}\Theta_{n+1}.
\end{align*}
Since $\sigma_n<|\omega_n|$ by definition of the sequence $\sigma_n$ (see Theorem~\ref{thm unif}) and $|\omega_n|\leq  C_1 |\omega|\lambda_n$ by Lemma~\ref{mcf estimates}, the time-$t$ map $\phi^t_{\lambda_n^{-1}P_n^{*}X_n}$ converges to $\phi^t_{\omega}$ in the $C^0$-topology for every $t\in\Rr$. 
\end{proof}

\begin{theorem}[$C^\ell$ conjugacy]\label{thm: gevrey conjugacy}
If there exists $\ell\in\Nn$, $\eta>0$ and $C>0$ such that,
$$
\sum_{n=1}^\infty
\frac{\Theta_n}{(R_{n-1}-d^{\frac{s-1}{s}}R_n)^s R_n^{s\ell}} \leq C \eta^{s\ell},
$$
then $h:=\lim_nh_n$ is a $C^\ell$ diffeomorphism and moreover
\begin{equation}\label{eq:Gevrey h}
\|\partial^{\alpha}(h-\id)\|_{C^0}\leq C' \alpha!^{s} \eta^{2s\ell},\quad  |\alpha|= \ell
\end{equation}
where $C'>0$ is independent of $\ell$.
\end{theorem}

\begin{proof}
Define
$$
D_n:=\frac{\Theta_n}{(R_{n-1}-d^{\frac{s-1}{s}}R_n)^s R_n^{s\ell}}.
$$
By hypothesis, there exists $\ell\in\Nn$ such that $\sum_n D_n<\infty$. 
From the definition of the $\CC_{R_n}$-norm we have for each $\alpha\in\Nn_0^d$ that
$$
\|\partial^{\alpha}(h_n-h_{n-1})\|_{C^0}\leq \frac{\alpha!^s}{R_n^{s|\alpha|}}\|h_n-h_{n-1}\|_{\CC_{R_n}}.
$$
So, by \eqref{estimate h_n-h_n-1} we get 
$$
\|\partial^{\alpha}(h_n-h_{n-1})\|_{C^0}\leq 8(C_\nu-1)\Gamma\alpha!^sD_n,\quad |\alpha|\leq \ell,
$$
where $\Gamma=\sup_n\Gamma_n$ and
$$
\Gamma_n:=(R_{n-1}-d^{\frac{s-1}{s}}R_n)^s+ 
2^s
\sum_{i=1}^{n-1}
\|g_i-\id\|_{\CC_{R_i}}.
$$
By the hypothesis of the theorem we conclude that $\Gamma\leq C'\eta^{s\ell}$ for some constant $C'>0$ independent of $\ell$. Since $\sum_n D_n<\infty$, the sequence $h_n-\id$ is Cauchy in $C^\ell(\Tt^d,\Rr^d)$. Thus, $h-\id\in C^\ell(\Tt^d,\Rr^d)$ where $h:=\lim_n h_n$. Taking in consideration Lemmas~\ref{lemma Hn-id vf v1} and ~\ref{lemma Hn-Hn-1 v1} we obtain, for any $m\geq1$ sufficiently large, that,
$$
\|g_m\circ\cdots\circ g_n -\id\|_{\CC_{R_n}}\leq \sum_{i=m}^n\|g_i-\id\|_{\CC_{R_i}},
\quad
n\geq m.
$$
In view of Lemma~\ref{lemma Wn-id} and $\sum_n\Theta_n<\infty$, the previous estimate gives $\|h-\id\|_{C^1}<1$. Thus, $h$ is a diffeomorphism. Let $|\alpha|=\ell$. To get the final estimate we write using a telescopic argument,
\begin{align*}
\|\partial^{\alpha}(h_n-\id)\|_{C^0}&\leq \sum_{i=1}^n\|\partial^{\alpha}(h_i-h_{i-1})\|_{C^0}\\
&\leq 8(C_\nu-1)\Gamma\alpha!^s\sum_{i=1}^n D_i\\
&\leq 8(C_\nu-1)C'C\alpha!^s\eta^{2s\ell}, 
\end{align*}
where we have assumed for convenience $h_0=\id$.
\end{proof}

\subsection{Sufficient conditions}
\label{sec: Sufficient conditions}

Define $R_0:=\rho_0-\lim\BB_n$, recalling~\eqref{eq defn BB_n}, and
$$
R_n:=\frac12\min\left\{\frac{R_0}{\beta^{2n}\Omega_n^{1/s}},\frac{R_{n-1}}{d^{\frac{s-1}{s}}}\right\},\quad n\geq 1,
$$
where,
$$
\Omega_n:=\max_{1\leq i \leq n} A_1\cdots A_{i-1}B_{i-1}.
$$
Notice that $\Omega_n\leq \Omega_{n+1}$. For convenience we set $\Omega_0=1$.
\begin{lemma}\label{lem:Rn} For every $n\geq 1$, ~\eqref{eq defn R_n} holds and
$$ \frac{R_0}{2^n\beta^{2n}\Omega_n^{1/s}}\leq 
R_n\leq 
\frac{R_0}{2\beta^{2n}\Omega_n^{1/s}}.
$$
\end{lemma}

\begin{proof}
Using \eqref{rhon} we see that
$$
\rho_{n-1}-\nu>\frac{\rho_0-\lim\BB_n}{\beta^{2(n-1)}A_1^{1/s}\cdots A_{n-1}^{1/s}}.
$$
Hence,
$$
r_n=\frac{\rho_{n-1}-\nu}{2\beta^2B_{n-1}^{1/s}}>\frac{R_0}{2\beta^{2n}A_1^{1/s}\cdots A_{n-1}^{1/s}B_{n-1}^{1/s}}\geq R_n.
$$
This shows the first inequality in ~\eqref{eq defn R_n}.
The other one is immediate from the definition of $R_n$.

Finally, the last inequalities follow by induction on $n$.
\end{proof}

In the following we give a sufficient condition for the conjugacy $h$ to have $C^\ell$-smooth regularity in terms of the growth of the sequence $t_n$. Recall that $W_n=W(t_n)$ and $\Delta_n=\tau_{k_n}-W(\tau_{k_n})$. Define
\begin{equation}\label{def:sigma}
\sigma_n := n^{-1}C_1^{-1}\mu^{-1} e^{-(d-1)(t_{n+1} -t_{n})-(d-1)W_n-t_{n+1}+\Delta_{n}}.
\end{equation}
and
\begin{equation}
\label{eq defn mu}
\mu:=(2\beta^2)^sC_2.
\end{equation}

\begin{proposition}\label{lem:Cl conjugacy}
Let $\ell\in\Nn$. If 
$$
t_{n+1}\geq 
5(\ell+1) t_n,
$$
then there are constants $C,\eta>0$ not depending on $\ell$ such that~\eqref{eq:Gevrey h} holds.
\end{proposition}

\begin{proof}
By Lemma~\ref{lem:Rn}
$$
(R_{n-1}-d^{\frac{s-1}{s}}R_n)^s>\frac{R_{n}^s}{2^s},\quad\forall\,n\geq 1.
$$
Moreover, from the definition of $\Theta_n$ and $\varepsilon_n$ (see Theorem~\ref{thm unif}) we have that,
$$
\Theta_n\leq \frac{|P_{n-1}^{-1}|}{\sigma_{n-1}}\varepsilon_{n-1}\leq \frac{|P_{n-1}^{-1}|\sigma_{n-1}}{C_\nu-1},\quad\forall\,n\geq1.
$$
Thus, by Lemma~\ref{lem:Rn},
\begin{align*}
\frac{\Theta_n}{(R_{n-1}-d^{\frac{s-1}{s}}R_n)^s R_n^{s\ell}}&\leq \frac{2^s\Theta_n}{R_n^{s(\ell+1)}}\\
&\leq \frac{2^s|P_{n-1}^{-1}|\sigma_{n-1}}{(C_\nu-1)R_n^{s(\ell+1)}}\\
&\leq c_1(2\beta^2)^{ns(\ell+1)}\Omega_n^{\ell+1}|P_{n-1}^{-1}|\sigma_{n-1},
\end{align*}
where
$$
c_1:=\frac{2^s}{(C_\nu-1)R_0^{s(\ell+1)}}.
$$

By Lemma~\ref{mcf estimates 2}, 
$$
\Omega_n\leq |\omega|C_1\left(\frac{C_2}{\mu}\right)^{n-1}\left(1+\frac1{n-1}\right)^{n-1} e^{(d+1)t_{n-1}}.
$$
Notice that
$$
\sigma_n\leq C_1^{-1}\mu^{-1} e^{-d(t_{n+1} -t_{n})}.
$$
Moreover, by Lemma~\ref{mcf estimates} and the definition of $\sigma_n$,
\begin{align*}
|P_{n-1}^{-1}|\sigma_{n-1}&\leq C_1^{-1}\mu^{-1}  C_2|\omega| \, e^{(d-1)(t_{n-1}-W_{n-1})+ dW_{n-1}-d(t_{n}-t_{n-1})}\\
&\leq C_1^{-1}\mu^{-1} C_2|\omega|\, e^{-d t_{n}+ 2d t_{n-1}}.
\end{align*}
Putting these estimates together we get
$$
\Omega_n^{\ell+1}|P_{n-1}^{-1}|\sigma_{n-1}
\leq 
e^{\ell+1}(C_1\mu C_2^{-1})^\ell|\omega|^{\ell+2} (\mu^{-1}  C_2)^{n(\ell+1)}e^{-d t_{n}+ 4d(\ell+1) t_{n-1}}.
$$
Thus,
$$
\frac{\Theta_n}{(R_{n-1}-d^{\frac{s-1}{s}}R_n)^s R_n^{s\ell}}
\leq 
c_2 a_n
$$
where
$$
c_2:=c_1e^{\ell+1}(C_1\mu C_2^{-1})^\ell|\omega|^{\ell+2}
$$
and, using the definition of $\mu$,
\begin{align*}
a_n:=&((2\beta^2)^s C_2\mu^{-1})^{n(\ell+1)}e^{-d t_{n}+ 4d(\ell+1) t_{n-1}} \\
=&
e^{-d (t_{n} - 4(\ell+1)t_{n-1})}.
\end{align*}
By the hypothesis on $t_n$ we conclude that $a_n\leq e^{-d(\ell+1)t_{n-1}}$ and $t_{n}\geq t_1(5(\ell+1))^{n-1} \geq t_1 n$. 
Hence,
\begin{align*}
\sum_{n=1}^\infty \frac{\Theta_n}{(R_{n-1}-d^{\frac{s-1}{s}}R_n)^s R_n^{s\ell}}
&\leq 
c_2\sum_{n=1}^\infty a_n \\
&\leq
c_2\sum_{n=1}^\infty e^{-d(\ell+1)t_1(n-1)} \\
&= \frac{c_2 e^{d(\ell+1)t_1}}{e^{d(\ell+1)t_1}-1} \\
&\leq
\frac{c_2 e^{d(\ell+1)t_1}}{e^{2dt_1}-1}. 
\end{align*}
\end{proof}

\subsection{Class of frequency vectors}

Recall that the numbers $A_n$, $\eta_n$ and $|T_n|$ depend on the choice of a strictly increasing unbounded sequence $t_n$.

\begin{lemma}\label{condition W renormalization}
If 
\begin{equation}
\label{eq defn sum t_n}
\sum_{n=0}^\infty e^{-\frac{1}{s}\Delta_n}t_{n+1}<\infty,
\end{equation}
then $\lim \BB_n<\infty$.
\end{lemma}

\begin{proof}
It follows from the definition of $\varepsilon_n$ (see \eqref{eq:espilon}) that
$$
\frac{\varepsilon_{n-1}}{\varepsilon_n}=\left(\frac{\sigma_{n-1}}{\sigma_n}\right)^2,
$$
for every $n\geq1$ sufficiently large. Moreover, from the definition of sequences $\sigma_n$ and $\theta_n=1$ we have that
$$
\log\left( \frac{\varepsilon_{n-1}}{\varepsilon_n\theta_n}\right)\leq 2d t_{n+1}+\log\theta_n^{-1}+2\log\left(\frac{n}{n-1}\right)\leq c_0 t_{n+1},
$$
for every $n$ sufficiently large and some constant $c_0>0$ independent of $n$.
Moreover, by Lemma~\ref{mcf estimates},
$$
\log(|\eta_n||T_n|)\leq \log(C_1C_2)+2d t_n.
$$
Hence, 
$$
\log\phi_n\leq  c_1 t_{n+1},
$$
for every $n$ sufficiently large and some constant $c_1>0$ independent of $n$.
Now the claim follows since by Lemma~\ref{mcf estimates 2}, we have that
$$
A_1^{1/s}\cdots A_n^{1/s}\leq (1+1/n)^\frac{n}{s}(\mu^{-1}C_2)^\frac{n}{s}e^{-\frac{1}{s}\Delta_n}
\leq 
c_2 \,e^{-\frac{1}{s}\Delta_n},
$$
for some constant $c_2>0$ independent of $n$.
\end{proof}

\subsection{Gevrey conjugation}

Notice that if $t_n=\tau_n$, the sequence of stopping times for $\omega$, then~\eqref{eq defn sum t_n} holds for $s$-Brjuno vectors.

The following theorem is the main result of this paper.

\begin{theorem}\label{th:main}
Let $X\in\FF'_{\rho_0}$ be an $s$-Gevrey vector field such that $\rot X=\vecomega=(\alpha,1)$ is an $s$-Brjuno vector. There is $\varepsilon_0=\varepsilon_0(\rho_0,\omega,s,d)>0$ such that if $\|X-\omega\|'_{\rho_0} < \varepsilon_0$, then there exists an $s$-Gevrey diffeomorphism $h$ such that $\phi_X^t\circ h=h\circ \phi_\omega^t$.
\end{theorem}

\begin{proof}
Denote by $\BB(\omega)$ the limit of $\BB_n$ for a vector $\omega$.
As in \cite[Theorem~8.1]{jld10} we can iterate the renormalization operator a finite number of steps $N\geq1$ to get $\rho_N>\BB(\omega_N)$. Therefore, we assume from the very beginning that $\rho_0>\BB(\omega)$. Notice that $\BB(\omega)<\infty$ by Lemma~\ref{condition W renormalization}. We now apply Theorem~\ref{convergence Rn} with $\sigma_n$ defined in \eqref{def:sigma} and $\theta_n=1$ to conclude that $X$ is infinitely renormalizable. The value of $\varepsilon_0$ is given by \eqref{eq:espilon}.

Define the following modified sequence. Let $\tilde{t}_1=t_1$ and
$$
\tilde{t}_{n+1}=\max\left\{5(\ell+1) \tilde{t}_n,t_{n+1}\right\},
\quad n\geq1.
$$
Notice that $\tilde{t}_n\geq t_n$ and $\tilde{t}_n$ satisfies the assumption of Proposition~\ref{lem:Cl conjugacy}. Moreover, if $\tilde{t}_{n+1}=t_{n+1}$, then
$$
e^{-\frac1s\Delta(\tilde{t}_n)}\tilde{t}_{n+1}\leq e^{-\frac1s\Delta_n}t_{n+1},
$$
since $\Delta_n\leq \Delta(\tilde{t}_n)$.
On other hand, if $\tilde{t}_{n+1}=5(\ell+1) \tilde{t}_n\geq t_{n+1}$, then
$$
e^{-\frac1s\Delta(\tilde{t}_n)}\tilde{t}_{n+1}\leq c\, e^{-\frac1s\Delta(\tilde{t}_n)}\tilde{t}_n\leq c\, e^{-\frac1s\Delta_{m_n}}t_{m_n+1},
$$
where $m_n=\max\{j\in\Nn_0\colon t_j\leq \tilde{t}_n\}$ and $c>0$ is a constant independent of $n$. Notice that $m_n\geq n$ and $m_n\to\infty$ since $\tilde{t}_n$ is unbounded. In either case, we conclude that
$$
\sum_{n=0}^\infty e^{-\frac1s\Delta(\tilde{t}_n)}\tilde{t}_{n+1}<\infty.
$$
So, both the hypothesis of Theorem~\ref{convergence Rn} and Proposition~\ref{lem:Cl conjugacy} hold using the modified sequence $\tilde{t}_n$. By Theorems~\ref{th: top conjugacy} and \ref{thm: gevrey conjugacy}, the vector field $X$ is $C^\ell$ conjugated to the constant vector field $\omega$. The conjugacy $h$ has $s$-Gevrey estimates \eqref{eq:Gevrey h}. Since $\ell\in\Nn$ is arbitrary and the conjugacy $h$ is unique up to a composition with a translation, we conclude that $h$ is $s$-Gevrey smooth. 
\end{proof}


\appendix
\section{}
\label{appendixC}
\subsection{Proof of Theorem~\ref{thm unif}}
Define
$$
\delta:=\frac{8(C_\nu-1)}{\sigma}\varepsilon <\frac12.
$$ 
Let $X=\vecw + f$ where $f\in \FF'_{\rho}$. 
We seek a coordinate transformation $U=\id +u$ where $u$ belongs to
$$
\BB_\delta=\{u\in \Ii^-_\sigma\FF'_{\rho'}\colon \|u\|'_{\rho'}<\delta\}.
$$
Notice that
$$
U^*X=(I+Du)^{-1}(\vecw+f\circ (\id + u))\,.
$$
Since 
\begin{equation*}
\begin{split}
\frac{\left(\rho'+d^{\frac{s-1}{s}}(\rho''+\nu)\right)^s}{2^sd^{s-1}}-(\rho''+\nu)^s&\geq \frac{1}{\beta^s}\left(\left(\rho'-\frac\nu2\right)^s-(\rho'-\nu)^s\right)\\
&\geq \frac{\nu^s}{(2\beta)^s},
\end{split}
\end{equation*}
and $\delta\leq\frac{\nu^s}{(2\beta)^s}$, Proposition~\ref{prop:norm estimates} implies that we have a well defined operator $\GG\colon\BB_\delta\to \Ii^-_\sigma\FF_{\rho''}$ given by,
$$
\GG(u):= \Ii^-_\sigma(I+Du)^{-1}(\vecw+f\circ (\id + u)).
$$
Notice that, $\GG(0)=\Ii^-_\sigma X\in\Ii^-_\sigma\FF'_{\rho}$.

We want to find $u\in\BB_\delta$ such that $\GG(u)=0$. We solve this problem using a homotopy, i.e. we will look for a smooth family $u_t\colon[0,1]\to\BB_\delta$ satisfying the equation,
$$
\GG(u_t)=(1-t)\GG(0).
$$
Differentiating with respect to $t$ we conclude that $u_t$ has to satisfy the differential equation
$$
D\GG(u_t)\frac{du_t}{dt}=-\GG(0)\,.
$$
In order to solve this differential equation we invert $D\GG(u_t)$. The following lemmas provide the necessary estimates.

\begin{lemma}
If $u\in\BB_\delta$, then the derivative of $\GG$ at $u$ is a linear operator $D\GG(u)\colon\Ii^-_\sigma\FF'_{\rho'}\to \Ii^-_\sigma\FF_{\rho''}$ defined by
\begin{equation}\label{eq:DF}
h\mapsto\Ii^-_\sigma(I+Du)^{-1}\left[(Df)\circ U h-Dh(I+Du)^{-1}(\vecw+f\circ U)\right].
\end{equation}
\end{lemma}
\begin{proof}
See~\cite[Lemma 9.2]{jld} for the computation of the derivative. To see that $D\GG(u)h\in\Ii^-_\sigma\FF_{\rho''}$ for any $h\in \Ii^-_\sigma\FF'_{\rho'}$ just apply Proposition~\ref{prop:norm estimates}.
\end{proof}

\begin{lemma}\label{lem:DF0} If $\|f\|'_{\rho}<\varepsilon$, then $D\GG(0)^{-1}$ is a bounded linear operator from $\Ii^-_\sigma\FF_{\rho}$ to $\Ii^-_\sigma\FF'_{\rho'}$. Moreover 
$$
\|D\GG(0)^{-1}\|<\frac{4(C_\nu-1)}{\sigma}.
$$
\end{lemma}

\begin{proof}
Let $\LL_fh=Df\,h-Dh\,f$ and $D_\vecw h=Dh\,\vecw$. Then
$$
D\GG(0)h=\Ii^-_\sigma(\LL_f-D_\vecw)h.
$$
We wish to invert $D\GG(0)$ on $\Ii^-_\sigma\FF_{\rho}$, i.e. on elements in $\FF_\rho$ having only far from resonant modes. Formally,
$$
D\GG(0)^{-1}=(\Ii^-_\sigma(\LL_f-D_\vecw))^{-1}=D_\vecw^{-1}(\Ii_\sigma^-\LL_f D_\vecw^{-1}-\Ii)^{-1}.
$$
The inverse of $D_\vecw$ is a bounded linear operator from $\Ii^-_\sigma\FF_{\rho}$ to $\Ii^-_\sigma\FF'_{\rho'}$. Indeed, given $g\in \Ii^-_\sigma\FF_{\rho}$,
\begin{align*}
\|(D_\vecw)^{-1}g\|'_{\rho'}&=\sum_{\veck\in I^-_\sigma}\frac{1+|\veck|}{|\veck\cdot\vecw|}|g_\veck|e^{\rho'|\veck|^{1/s}}\\
&\leq \sum_{\veck\in I^-_\sigma}\frac{1+|\veck|}{\sigma|\veck|}|g_\veck|e^{(\rho-\nu)|\veck|^{1/s}}\\
&\leq \frac{2(C_\nu-1)}{\sigma} \|g\|_{\rho}\,.
\end{align*}
Moreover, $\LL_f$ is a bounded linear operator from $\Ii^-_\sigma\FF'_{\rho'}$ to $\FF_{\rho'}$,
$$
\|\LL_f h\|_{\rho'}\leq \|Df\,h\|_{\rho'}+\|Dh\,f\|_{\rho'}\leq 2\|f\|'_{\rho} \|h\|'_{\rho'}.
$$
Thus, 
$$
\|\Ii^-_\sigma\LL_f D_\vecw^{-1}\|\leq \frac{4(C_\nu-1)}{\sigma}\|f\|'_{\rho} <\frac{1}{2},
$$ 
since $\|f\|'_\rho<\varepsilon<\frac{\sigma}{8(C_\nu-1)}$. Hence, 
$$
\|D\GG(0)^{-1}\|\leq\frac{\|D_\vecw^{-1}\|}{1-\|\Ii_\sigma^-\LL_f D_\vecw^{-1}\|}< \frac{4(C_\nu-1)}{\sigma} .
$$
\end{proof}

\begin{lemma}\label{lem:DFu0}
If $u\in\BB_\delta$ and $\|f\|'_\rho<\varepsilon$, then the linear operator $D\GG(u)-D\GG(0)$ mapping $\Ii_\sigma^-\FF'_{\rho'}$ to $\Ii_\sigma^-\FF_{\rho''}$ is bounded and
$$
\|D\GG(u)-D\GG(0)\|< \frac{\delta|\vecw| C_\nu}{C_\nu-1}\left(\frac{2^s}{\nu^s}+7\right)
$$
\end{lemma}
\begin{proof}
According to \eqref{eq:DF} we can write
$$
\left(D\GG(u)-D\GG(0)\right)h=\Ii_\sigma^-(I+Du)^{-1}\left(A_1+A_2+A_3\right),
$$
where 
\begin{align*}
A_1&=\left(Df\circ(\id+u)-Df-Du\,Df\right)h,\\
A_2&=Du\,Dh(\vecw+f),\\
A_3&=-Dh(I+Du)^{-1}\left(f\circ(\id+u)-f-Du(\vecw+f)\right).
\end{align*}
It follows from Proposition~\ref{prop:norm estimates} that,
\begin{align*}
\|A_1\|_{\rho''}&\leq \left(\frac{2^s C_\nu}{\nu^s}+1\right)\|f\|'_{\rho}\|u\|'_{\rho'}\|h\|'_{\rho'},\\
\|A_2\|_{\rho''}&\leq (|\vecw| + \|f\|'_{\rho})\|u\|'_{\rho'}\|h\|'_{\rho'},\\
\|A_3\|_{\rho''}&\leq  \frac{\|u\|'_{\rho'}}{1-\|u\|'_{\rho'}}\left[|\vecw|+ \left(1 +  C_\nu\right)\|f\|'_{\rho}\right]\|h\|'_{\rho'}.
\end{align*}
Taking into account that $\|u\|'_{\rho'}<\delta< 1/2$, $\|f\|'_\rho<\varepsilon<\frac{\sigma}{8(C_\nu-1)}$ and $0<\sigma<|\vecw|$ we get,
$$
\|D\GG(u)-D\GG(0)\|< \frac{|\vecw|\delta C_\nu}{4(C_\nu-1)}\left(\frac{2^s}{\nu^s}+26\right)
$$
which gives the final estimate.
\end{proof}

\begin{lemma}\label{lem:DFu}
If $u\in\BB_\delta$ and $\|f\|'_\rho<\varepsilon$, then $D\GG(u)^{-1}$ is a bounded linear operator from $\Ii_\sigma^-\FF_{\rho}$ to $\Ii_\sigma^-\FF'_{\rho'}$. Moreover,
\begin{align*}
\|D\GG(u)^{-1}\|<\frac\delta\varepsilon.
\end{align*}
\end{lemma}

\begin{proof}
Notice that,
\begin{align*}
D\GG(u)^{-1}&=(D\GG(u)-D\GG(0)+D\GG(0))^{-1} \\
&=D\GG(0)^{-1}\left[I+(D\GG(u)-D\GG(0))D\GG(0)^{-1}\right]^{-1}.
\end{align*}
By Lemmas~\ref{lem:DF0} and \ref{lem:DFu0},
\begin{align*}
\|D\GG(u)-D\GG(0)\|\|D\GG(0)^{-1}\|&<\frac{4\delta|\vecw| C_\nu}{\sigma}\left(\frac{2^s}{\nu^s}+7\right)\\
&<\frac12,
\end{align*}

by our choice of $\delta$. Thus, again using Lemma~\ref{lem:DF0}
$$
\|D\GG(u)^{-1}\|<2\|D\GG(0)^{-1}\|<\frac{8(C_\nu-1)}{\sigma} =\frac{\delta}{\varepsilon}.
$$
\end{proof}

Now we conclude the proof of Theorem~\ref{thm unif}. Notice that,
$$
u_t=-\int_0^tD\GG(u_s)^{-1}\GG(0)\,ds.
$$
Since $\GG(0)\in\Ii_\sigma^-\FF_\rho$, it follows from Lemma~\ref{lem:DFu} that, 
\begin{equation}\label{norm of u}
\|u_t\|'_{\rho'}\leq t \sup_{u\in\BB_\delta}\|D\GG(u)^{-1}\|\|\GG(0)\|_\rho<\frac{8t(C_\nu-1)}{\sigma}\|\Ii_\sigma^- X\|_\rho.
\end{equation}
This implies that $u_t\in\BB_\delta$ for every $t\in[0,1]$. So $X\mapsto u_t$ defines an operator $\fU_t$ from $\VV_\varepsilon$ to $\Ii^-_\sigma\FF'_{\rho'}$ and $X\mapsto (\id+\fU_t(X))^*X$ defines another operator $\UU_t$ from $\VV_\epsilon$ to $(1-t)\Ii^-_\sigma\FF'_{\rho}\oplus t\Ii^+_\sigma\FF_{\rho''}$.
In addition, 
$$
\UU_t(\vecw+f)-\vecw=\Ii^+_\sigma f+(1-t)\Ii^-_\sigma f+\Ii^+_\sigma(A_1 +A_2+A_3)
$$
where 
\begin{align*}
A_1&=Df\,u_t-Du_tf-Du_tDf\,u_t,\\
A_2&=\left(I-Du_t\right)\left(f\circ(\id+u_t)-f-Df\,u_t\right),\\
A_3&=\sum_{n=2}^\infty(-Du_t)^n\left(\vecw+f\circ(\id+u_t)\right).
\end{align*}
Using \eqref{norm of u} and Proposition~\ref{prop:norm estimates} we get,
\begin{align*}
\|A_1\|_{\rho''}&\leq \frac{24t(C_\nu-1)}{\sigma}{\|f\|'_\rho}^2,\\
\|A_2\|_{\rho''}&\leq \frac{32tC_\nu(C_\nu-1)}{\sigma}{\|f\|'_\rho}^2,\\
\|A_3\|_{\rho''}&\leq \frac{2^7t|\vecw|(C_\nu-1)(2C_\nu-1)}{\sigma^2}{\|f\|'_\rho}^2.
\end{align*}
Therefore, $\UU_t$ is Fr\'echet differentiable at $\vecw$ with derivative $\Ii^+_\sigma f+(1-t)\Ii^-_\sigma f$ and the estimates in the statement follow immediately. This concludes the proof of Theorem~\ref{thm unif}.

\section*{Acknowledgements}

The authors were partially supported by the Project CEMAPRE - UID/MULTI/00491/2013 financed by FCT/MCTES through national funds.
JPG was also supported by the postdoctorial fellowship SFRH/BPD/78230/2011 funded by FCT/MCTES.

The authors are grateful for the comments and suggestions from participants in the `10th AIMS Conference on Dynamical Systems, Differential Equations and Applications'' (Madrid, July 2014) and ``The Dynamics of Complex Systems: A meeting in honour of the 60th birthday of Robert MacKay'' (University of Warwick, May 2016) where this work was presented.


\bibliographystyle{plain} 
\bibliography{rfrncs}

\begin{thebibliography}{10}

\bibitem{Arnold2}
V.~I. Arnol'd.
\newblock Small denominators {I}, mappings of the circumference onto itself.
\newblock {\em Transl. {AMS} 2nd Series}, 46:213--284, 1961.

\bibitem{Bekka2000}
M.~B. Bekka and M.~Mayer.
\newblock {\em Ergodic theory and topological dynamics of group actions on
  homogeneous spaces}, volume 269 of {\em London Mathematical Society Lecture
  Note Series}.
\newblock Cambridge University Press, Cambridge, 2000.

\bibitem{MR2835876}
A.~Bounemoura.
\newblock Effective stability for {G}evrey and finitely differentiable
  prevalent {H}amiltonians.
\newblock {\em Comm. Math. Phys.}, 307(1):157--183, 2011.

\bibitem{MR1765828}
T.~Carletti and S.~Marmi.
\newblock Linearization of analytic and non-analytic germs of diffeomorphisms
  of {$({\bf C},0)$}.
\newblock {\em Bull. Soc. Math. France}, 128(1):69--85, 2000.

\bibitem{C11}
Y.~Cheung.
\newblock Hausdorff dimension of the set of singular pairs.
\newblock {\em Ann. of Math.}, 173:127--167, 2011.

\bibitem{C13}
N.~Chevallier.
\newblock Best simultaneous diophantine approximations and multidimensional
  continued fraction expansions.
\newblock {\em Moscow J. of Combinatorics and Number Theory}, 3:3--56, 2013.

\bibitem{Herman1979}
M.~R. Herman.
\newblock Sur la conjugaison {diff\'{e}rentiable} des {diff\'{e}omorphismes} du
  cercle {\`{a}} des rotations ({O}n the differentiable conjugacy of the
  diffeomorphims of the circle to rotations).
\newblock {\em Publ. Math. Inst. Hautes {\'{E}tud.} Sci.}, 49:5--233, 1979.

\bibitem{jld5}
K.~Khanin, J.~Lopes Dias, and J.~Marklof.
\newblock Multidimensional continued fractions, dynamic renormalization and
  {KAM} theory.
\newblock {\em Comm. Math. Phys.}, 207:197--231, 2007.

\bibitem{jld9}
K.~Khanin, J.~Lopes~Dias, and J.~Marklof.
\newblock Renormalization of multidimensional {H}amiltonian flows.
\newblock {\em Nonlinearity}, 19:2727--2753, 2006.

\bibitem{Koch}
H.~Koch.
\newblock A renormalization group for {H}amiltonians, with applications to
  {KAM} tori.
\newblock {\em Erg. Theor. Dyn. Syst.}, 19:475--521, 1999.

\bibitem{Koch2004}
H.~Koch.
\newblock A renormalization group fixed point associated with the breakup of
  golden invariant tori.
\newblock {\em Discrete Contin. Dyn. Syst.}, 11:881--909, 2004.

\bibitem{MR2679012}
H.~Koch and S.~Koci\'c.
\newblock A renormalization approach to lower-dimensional tori with {B}rjuno
  frequency vectors.
\newblock {\em J. Differential Equations}, 249(8):1986--2004, 2010.

\bibitem{Koch-Kocic06}
H.~Koch and S.~Koci\'c.
\newblock A renormalization group approach to quasiperiodic motion with
  {B}rjuno frequencies.
\newblock {\em Erg. Theor. Dyn. Syst.}, 30:1131--1146, 2010.

\bibitem{jld8}
H.~Koch and J.~Lopes~Dias.
\newblock Renormalization of diophantine skew flows, with applications to the
  reducibility problem.
\newblock {\em Discrete Contin. Dyn. Syst.}, 21:477--500, 2008.

\bibitem{Kocic2007}
S.~Koci\'c.
\newblock Reducibility of skew-product systems with multidimensional {B}rjuno
  base flows.
\newblock {\em Discrete Contin. Dyn. Syst.}, 29(1):261--283, 2011.

\bibitem{Lagarias94}
J.~C. Lagarias.
\newblock Geodesic multidimensional continued fractions.
\newblock {\em Proc. London Math. Soc.}, 69:464--488, 1994.

\bibitem{jld}
J.~Lopes~Dias.
\newblock Renormalization of flows on the multidimensional torus close to a
  {$KT$} frequency vector.
\newblock {\em Nonlinearity}, 15:647--664, 2002.

\bibitem{jld7}
J.~Lopes~Dias.
\newblock A normal form theorem for {B}rjuno skew-systems through
  renormalization.
\newblock {\em J. Differential Equations}, 230:1--23, 2006.

\bibitem{jld10}
J.~Lopes~Dias.
\newblock Local conjugacy classes for analytic torus flows.
\newblock {\em J. Differential Equations}, 2008.

\bibitem{MacKayThesis}
R.~S. MacKay.
\newblock {\em Renormalisation in area-preserving maps}.
\newblock World Scientific Publishing Co. Inc., River Edge, NJ, 1993.

\bibitem{MacKay}
R.~S. MacKay.
\newblock Three topics in {H}amiltonian dynamics.
\newblock In Y.~Aizawa, S.~Saito, and K.~Shiraiwa, editors, {\em Dynamical
  Systems and Chaos}, volume~2. World Scientific, 1995.

\bibitem{MS02}
J.-P. Marco and D.~Sauzin.
\newblock Stability and instability for {G}evrey quasi-convex near-integrable
  {H}amiltonian systems.
\newblock {\em Publ. Math. Inst. Hautes \'Etudes Sci.}, (96):199--275 (2003),
  2002.

\bibitem{MR2684068}
T.~Mitev and G.~Popov.
\newblock Gevrey normal form and effective stability of {L}agrangian tori.
\newblock {\em Discrete Contin. Dyn. Syst. Ser. S}, 3(4):643--666, 2010.

\bibitem{MR2104602}
G.~Popov.
\newblock K{AM} theorem for {G}evrey {H}amiltonians.
\newblock {\em Ergodic Theory Dynam. Systems}, 24(5):1753--1786, 2004.

\bibitem{rodino}
L.~Rodino.
\newblock {\em Linear partial differential operators in {G}evrey spaces}.
\newblock World Scientific Publishing Co., Inc., River Edge, NJ, 1993.

\bibitem{R01}
H.~Russmann.
\newblock Invariant tori in non-degenerate nearly integrable hamiltonian
  systems.
\newblock {\em Regul. Chaotic Dyn.}, 6:119--204, 2001.

\bibitem{MR0007044}
Carl~Ludwig Siegel.
\newblock Iteration of analytic functions.
\newblock {\em Ann. of Math. (2)}, 43:607--612, 1942.

\bibitem{MR2684071}
F.~Wagener.
\newblock A parametrised version of {M}oser's modifying terms theorem.
\newblock {\em Discrete Contin. Dyn. Syst. Ser. S}, 3(4):719--768, 2010.

\bibitem{MR2525200}
X.~Wang and J.~Xu.
\newblock Gevrey-smoothness of invariant tori for analytic reversible systems
  under {R}\"ussmann's non-degeneracy condition.
\newblock {\em Discrete Contin. Dyn. Syst.}, 25(2):701--718, 2009.

\bibitem{MR2317497}
J.~Xu and J.~You.
\newblock Gevrey-smoothness of invariant tori for analytic nearly integrable
  {H}amiltonian systems under {R}\"ussmann's non-degeneracy condition.
\newblock {\em J. Differential Equations}, 235(2):609--622, 2007.

\bibitem{Yoccoz2}
J.-C. Yoccoz.
\newblock Petits diviseurs en dimension 1 ({S}mall divisors in dimension one).
\newblock {\em Ast\'erisque}, 231, 1995.

\bibitem{Yoccoz4}
J.-C. Yoccoz.
\newblock Analytic linearization of circle diffeomorphisms.
\newblock In Marmi and Yoccoz, editors, {\em Dynamical systems and small
  divisors}, volume 1784 of {\em Lecture Notes in Mathematics}.
  Springer-Verlag, 2002.

\bibitem{MR0380867}
E.~Zehnder.
\newblock Generalized implicit function theorems with applications to some
  small divisor problems. {I}.
\newblock {\em Comm. Pure Appl. Math.}, 28:91--140, 1975.

\bibitem{MR2257153}
D.~Zhang and J.~Xu.
\newblock On elliptic lower dimensional tori for {G}evrey-smooth {H}amiltonian
  systems under {R}\"ussmann's non-degeneracy condition.
\newblock {\em Discrete Contin. Dyn. Syst.}, 16(3):635--655, 2006.

\end{thebibliography}

\end{document}